\documentclass[]{scrartcl}
\usepackage{amsmath}
\usepackage[english]{babel}
\usepackage{lmodern}
\usepackage[final]{microtype} 
\usepackage{graphicx}
\usepackage{subcaption}
\usepackage{tikz}
\usetikzlibrary{calc, fit}
\usepackage{tumcolors2}
\usepackage{pgfplots}
\usepackage{kbordermatrix}
\renewcommand{\kbldelim}{(}
\renewcommand{\kbrdelim}{)}
\usepackage{mathtools}

\usepackage[%
  allbordercolors=tumblue1,%
  allcolors=tumblue1,%
  breaklinks=true, 
  colorlinks=false, 
  plainpages=false,%
  hypertexnames=false,%
  bookmarksopen=false,%
  bookmarksnumbered=false%
]%
{hyperref}%
\usepackage[autostyle=true]{csquotes}
\usepackage[backend=biber, citestyle=alphabetic, bibstyle=alphabetic,
hyperref=true, maxbibnames=20, maxcitenames=3, url=false, doi=false,
giveninits=true]{biblatex}
\usepackage{amsmath, amssymb}
\usepackage[hyperref, amsmath, thmmarks]{ntheorem}
\usepackage{tabularx}
\usepackage{booktabs}
\usepackage{caption}
\usepackage{subcaption}
\usepackage{pifont}
\usepackage[hardnesscommands=bb]{m9macros}
\usetikzlibrary{calc, shapes, decorations, automata, positioning, external, arrows, spy, intersections, matrix}
\usepackage{enumitem}
\setlist[enumerate,1]{label=\alph*)}
\setlist[enumerate,2]{label=(\roman*),ref=\theenumi(\roman*)}
\setlist[enumerate,3]{label=\Roman*.}

\usepackage{nicefrac}
\usepackage{bbm}

\usepackage{algorithm}
\usepackage[]{algpseudocode}
\let\oldState\State
\newcommand*{\stopnumbering}{%
  \let\olditem\item
  \renewcommand{\item}[1][]{\olditem[]}%
  \let\State\Statex}
\newcommand*{\resumenumbering}{%
  \let\item\olditem
  \let\State\oldState}
\algrenewcommand\algorithmicrequire{\textbf{Input:}}
\algrenewcommand\algorithmicensure{\textbf{Output:}}
\newcommand{\posreals}{\ensuremath{\reals_{\geq 0}}}

\newcommand{\reals}{\ensuremath{\mathds{R}}}
\newcommand{\naturals}{\ensuremath{\mathbb{N}}}

\newcommand{\ints}[1]{\ensuremath{[#1]}}

\newcommand{\inn}{\text{in}}
\newcommand{\outn}{\text{out}}

\let \oldforall\forall
\renewcommand{\forall}{\oldforall\ }		
\newcommand{\cupdot}{\mathbin{\dot{\cup}}}

\newcommand*\colvec[1]{\begin{pmatrix}#1\end{pmatrix}}
\newcommand*\scolvec[1]{\begin{psmallmatrix}#1\end{psmallmatrix}}
\renewcommand{\phi}{\varphi}
\renewcommand{\epsilon}{\varepsilon}

   \providecommand\given{}
   \newcommand\SetSymbol[1][]{%
      \nonscript\:#1\vert
      \allowbreak
      \nonscript\:
      \mathopen{}}
   \DeclarePairedDelimiterX\Set[1]\{\}{%
      \renewcommand\given{\SetSymbol[\delimsize]}
      #1
   }

\makeatletter
\def\myspecialnode#1{
	\tikz@scan@one@point\pgfutil@firstofone($#1$)
	\pgf@xa=\pgf@x%
	\pgf@ya=\pgf@y%
	\pgfmathparse{1/(\pgf@xa*\pgf@xa/(\pgf@xx*\pgf@xx)+\pgf@ya*\pgf@ya/(\pgf@yy*\pgf@yy))}
	\let\myval=\pgfmathresult
}
\makeatother

\newcommand{\DC}{DC model\xspace}

\newcommand{\TR}{Transport model\xspace}

\newcommand{\tension}{elasticity\xspace}

\newcommand{\fmax}{{\bar{f}}}
\newcommand{\fu}{f^+}
\newcommand{\fl}{f^-}
\newcommand{\pmax}{{\bar{p}}}
\newcommand{\pu}{p^+}
\newcommand{\pl}{p^-}

\newcommand{\QDCf}{Q_f}

	\tikzset{
		highlight/.style={
			fill=accentuating green,
			fill opacity=.2},
		gridnode/.style={
			circle,
			minimum size=18pt,
			inner sep=0pt,
			text=black,
			draw=black},
		plainplot/.style={
			no markers,
			thick},
		comptime/.style={
			const plot,
			plainplot,
			solid},
		extpoint/.style={
			fill,
			circle,
			minimum width=4pt,
			inner sep=0},
		transparent label/.style={
			fill=white,
			fill opacity=0.6,
			text opacity=1}
	}

	\pgfplotsset{
		compat=1.14,
		cleanplot/.style={
			width=.7\textwidth,
			axis lines=middle,
			xtick=\empty,
			ytick=\empty,
			every axis x label/.style={at={(current axis.right of origin)},anchor=north east},
			every axis y label/.style={at={(current axis.above origin)},anchor=north east}},
	}

\usepackage[ngerman,english, capitalize]{cleveref}

\crefname{enumi}{}{}
\crefname{equation}{}{}
\creflabelformat{enumi}{#2#1#3}
\theoremnumbering{arabic}%
\theorembodyfont{\itshape}%
\theoremheaderfont{\bfseries}%
\theoremseparator{}%
\newtheorem{theorem}{Theorem}%
\crefname{theorem}{Theorem}{Theorems}
\newtheorem{proposition}[theorem]{Proposition}
\crefname{proposition}{Proposition}{Propositions}
\newtheorem{lemma}[theorem]{Lemma}%
\crefname{lemma}{Lemma}{Lemmas}
\newtheorem{corollary}[theorem]{Corollary}
\crefname{corollary}{Corollary}{Corollaries}

\theorembodyfont{\upshape}%
\newtheorem{remark}[theorem]{Remark}%
\crefname{remark}{Remark}{Remarks}
\newtheorem{definition}[theorem]{Definition}%
\crefname{definition}{Definition}{Definitions}

\theoremstyle{break}
\newtheorem{example}[theorem]{Example}%
\crefname{example}{Example}{Examples}

\theoremstyle{nonumberplain}%
\theoremheaderfont{\bfseries}%
\theorembodyfont{\normalfont}%
\theoremseparator{.}
\theoremsymbol{\ensuremath{\Box}}%
\newtheorem{proof}{Proof}%
\crefname{proof}{Proof}{Proofs}

\usepackage{thmtools}
\theoremstyle{break}
\makeatletter
\declaretheoremstyle[
style=break,
qed={},
headpunct={},
]{examplec}
\crefname{examplec}{Example}{Example}
\makeatother
\renewcommand\thmcontinues[1]{%
	\ifcsname hyperref\endcsname
		\hyperref[#1]{continued}%
	\else
		continued%
	\fi}

\hypersetup{
 pdfauthor={René Brandenberg, Paul Stursberg},
 pdftitle={Extremal solutions for Network Flow with Differential Constraints – A Generalization of Spanning Trees},
 pdfsubject={},
 pdfkeywords={},
 colorlinks=true, 
}

\author{René Brandenberg \and Paul Stursberg\thanks{supported by the German Federal Ministry for Economic Affairs and Energy (FKZ 03ET4029) on the basis of a decision by the German Bundestag.}\\Technische Universität München, Germany\\
\{rene.brandenberg, paul.stursberg\}@tum.de}
\title{Extremal solutions for Network Flow with Differential Constraints -- A Generalization of Spanning Trees}

\makeatletter
\def\blfootnote{\gdef\@thefnmark{}\@footnotetext}
\makeatother

\begin{document}

\maketitle
\blfootnote{\textup{2020} \textit{Mathematics Subject Classification}.
Primary 05C21; Secondary 90C35, 90B10, 05C83, 05C05.}

\begin{abstract}

  In network flow problems, there is a well-known one-to-one relationship between extreme points of the feasibility region and trees in the associated undirected graph. The same is true for the dual differential problem. In this paper, we study problems where the constraints of both problems appear simultaneously, a variant which is motivated by an application in the expansion planning of energy networks.
We show that all extreme points still directly correspond to graph-theoretical structures in the underlying network. The reverse is generally also true in all but certain exceptional cases.
We furthermore characterize graphs in which these exceptional cases never occur and present additional criteria for when those cases do not occur due to parameter values.

\end{abstract}


\section{Introduction}\label{sec:intro}


The \emph{minimum cost flow} problem is a classical optimization problem on a graph: In a network with cost and capacity bound assigned to its edges, what is the lowest-cost configuration to transfer a given amount of \enquote{flow} from a dedicated \emph{source} to a \emph{sink}?
This problem has been studied extensively and is one of the cornerstones of \emph{network flow theory} (see, \eg \cite{Rockafellar:1984}). It is well-known in particular for a very elegant combinatorial characterization of extremal solutions in terms of \emph{spanning trees} in the underlying graph.

One area of application where network flows commonly play a role is \emph{Energy System Optimization}, the study of optimal dispatch and expansion in electrical power systems. Especially in modern systems with a high share of renewables, transmission networks may play an important role since supply from renewable energy sources (\eg offshore or coastal wind turbines
) may be located far from demand centers.

These transmission systems typically operate under \emph{alternating current} and power flows across the system are governed by physical laws, whose non-convex equations do not lend themselves very well to mathematical optimization methods. As a consequence, different approximations have been developed to represent the most important aspects of an alternating current transmission grid. The simplest such approximation is the so-called \emph{Transport Model}. Here, the flow on each individual transmission line is restricted only by a capacity constraint, like in the classical network flow problem mentioned above (see, e.g., \cite{Tuohy:2009,Schaber:2012}).

While flows on different lines are independent under the Transport Model, the Linearized Load Flow Model (or \DC) approximates the interaction of flows on different transmission lines in the following way (see \cite{Stott:2009}): A \emph{potential value} is assigned to each vertex in the network and the power flow on each line is required to be proportional to the potential difference between its two endpoints. The constant of variation may be different for each individual line and depends on its technical parameters (the line's \emph{susceptance}), influenced by the length of the line and the material used (see, e.g., \cite{Hewes:2016}).

To enable a more targeted analysis, one needs a better understanding of the solution structures. These structures have long been known for the classical network flow theory which, as mentioned above, corresponds to the \TR. The \DC, in a similar way, can be captured mathematically by \emph{differential flows} in a network. In this paper, we generalize the theory of classical network flows to the setting of differential flows, providing a useful characterization of its extremal solutions. This work is based on results from \cite{Stursberg:2019}.

\subsection{Network graph and differential flow}

Starting from an undirected graph, we assign arbitrary directions to each edge in order to distinguish the direction of flow on that edge (such that a negative flow represents a flow in the reverse direction). The result is an \emph{anti-symmetric} directed graph (\ie between every pair of vertices there exists at most one directed edge). Adding upper and lower bounds for flows on each edge and for flow conservation deficit/surplus in each vertex, as well as an elasticity value for each edge, we obtain the following definition of a \emph{network} in the context of this article:

\begin{definition}[Network]\label{def:network_graph}
	A \emph{network} $(V,E,b,\pmax,\fmax)$ consists of a weakly connected, anti-symmetric directed graph $(V,E)$ together with an \emph{\tension vector} $b \in \reals^E$ with $b > 0$ and pairs of upper and lower bounds $\pmax = (\pl,\pu) \in (\reals \cup \Set{-\infty})^V \times (\reals \cup \Set{\infty})^V$ and $\fmax = (\fl,\fu) \in (\reals \cup \Set{-\infty})^E \times (\reals \cup \Set{\infty})^E$ for vertices and edges, respectively, such that $\pl \leq \pu$ and $\fl \leq \fu$.
\end{definition}

We write, \eg, $\pu=\infty$ to denote that $\pu_e = \infty$ for all $e \in E$.

Based on our motivation that positive and negative flows on an edge merely represent the same flow rate in different directions, it seems most natural to have $\fl_e = - \fu_e$ for all $e \in E$. However, all results in this paper hold for the more general case without this requirement, as well. The elasticity vector $b$ corresponds to the \emph{line susceptance} in an alternating current power transmission network.

\begin{definition}[Differential Flow]
	\label{def:feas_diff_flow}
	Let $G=(V,E,b,\pmax,\fmax)$ be a network. We call $f \in \reals^E$ a \emph{differential flow} in $G$ if there exists a potential $\phi \in \reals^V$ such that for $(v,w) \in E: f_{vw} = b_{vw} \cdot (\phi_w - \phi_v)$. In this case, we say that $\phi$ \emph{induces} $f$.
	Furthermore, we say that $f$ is \emph{feasible} (for $G$) if
	\begin{enumerate}
		\item $\fl_e \leq f_e \leq \fu_e$ for all $e \in E$, and
		\item $\pl_v \leq \sum_{e \in \delta^\outn(v)} f_e - \sum_{e \in \delta^\inn(v)} f_e \leq \pu_v$ for all $v \in V$,
	\end{enumerate}
	where $\delta^\inn(v)$ denotes the set of all edges ending in $v$, while $\delta^\outn(v)$ denotes the set of all edges starting in $v$.
\end{definition}

A differential flow is thus an ordinary flow on the edges of the graph $(V,E)$, which satisfies capacity
constraints on the edges and relaxed flow conservation constraints on the vertices, as well as the additional \emph{differential constraints} $f_{vw} = b_{vw} \cdot (\phi_w - \phi_v)$ for a suitable choice of the potential $\phi$. Note that the vertex constraints represent bounds on the flow that \emph{is created} in the respective vertex, the upper bound $\pu_v$ thus represents a \emph{lower} bound on the \emph{excess} in the vertex $v$ as defined, \eg, in \cite[Ch. 7.6]{Ahuja:1993}.

In the case of $\pl=\pu=0$ the conditions in b) would be ordinary flow conservation constraints (which are well-known from the literature on network flows, see, \eg, \cite[Chs. 3.4 and 7.3]{Papadimitriou:1998} or \cite[Ch. 13.2]{Schrijver:2003}) for all vertices. In that case all flows would be sums of flows along cycles. However, the differential constraints require that no cycle can carry a nonzero flow (since otherwise $\phi$ would strictly increase along the cycle). Consequently, $f \equiv 0$ would be the only feasible differential flow in $G$.

\subsection{The Differential Flow Polytope}
Given a \emph{network} $G=(V,E,b,\pmax,\fmax)$,
we define the \emph{\tension matrix} $B$ of $G$ by
\begin{equation}
  B_{ve} := \begin{cases}
    b_e & e\in \delta^\inn(v)\\
    -b_e & e\in \delta^\outn(v)\\
    0 & \text{else.}
  \end{cases}\label{eq:susceptance_matrix}
\end{equation}
Thus we have $B = A \cdot \diag(b)$ with $A$ being the incidence matrix of the graph $(V, E)$.

We can now write the set of feasible differential flows as
\begin{equation}
  \QDCf(G) := \Set*{f \in \reals^E \given \exists \phi \in \reals^V: \begin{gathered}B^{\top}\phi=f \\ \pl \leq -Af \leq \pu \\ \fl \leq f \leq \fu \end{gathered}}.\label{eq:q_f}
\end{equation}

As mentioned above, we refer to the first set of constraints $B^{\top}\phi=f$ as \emph{differential constraints} and to the other two sets of constraints $\pl \leq -Af \leq \pu$ and $\fl \leq f \leq \fu$ as \emph{relaxed flow conservation constraints} (or \emph{vertex constraints}) and \emph{capacity constraints} (or \emph{edge constraints}), respectively. We also see, as mentioned above, that inverting the sign of the vector $b$ does not change the resulting set $\QDCf(G)$, it merely inverts the sign of the corresponding $\phi$, as well.

Note that $\dim(\QDCf(G)) \leq |V|$, since the polyhedron is contained in the $|V|$-dimensional linear subspace defined by $B^{\top}\phi=f$. If we identify this linear subspace with $\reals^V$, we obtain the following alternative representation of $\QDCf(G)$:
\begin{equation}
  Q_\phi(G) := \Set*{\phi \in \reals^V \given \begin{gathered} \fl \leq B^{\top}\phi \leq \fu \\ \pl \leq -AB^{\top}\phi \leq \pu \end{gathered}}\label{eq:q_phi}
\end{equation}

The matrix $AB^\top$ is known in electrical engineering as the \emph{nodal admittance matrix} (see e.g., \cite{Stott:2009}) and if $b_e=1$ for all $e \in E$ then $AB^{\top} = AA^\top$ is known as the \emph{Laplacian matrix} of $(V,E)$, which captures many interesting properties of the graph (see, \eg, \cite{Chung:1997} about its eigenvalues).

Before we proceed, we observe that all three matrices $A$, $B$, and $AB^\top$ have rank $|V|-1$.

\begin{proposition}
  \label{prop:ABprop}
  Let $(V,E)$ be a weakly connected directed graph (\ie~one where the underlying undirected graph is connected) with incidence matrix $A$. Let $b \in \reals^E$, $b > 0$ and let $B$ be the \tension matrix as defined in \cref{eq:susceptance_matrix}.
  Then,
  \begin{equation} \label{eq:all-ones}
    A^\top v = 0 \Longleftrightarrow B^\top  v = 0 \Longleftrightarrow AB^\top  v = 0 \Longleftrightarrow v = \lambda \cdot \mathbbm{1}
    \text{ for } \lambda \in \reals,
  \end{equation}
  where $\mathbbm{1} : = \colvec{1\\\vdots\\1}$.
\end{proposition}

\subsection{Optimal Flows and Differentials}
\label{sec:flows_and_differentials}

The polyhedron $\QDCf(G)$ bears a close resemblance to the \emph{feasible flow} and \emph{feasible differential} polyhedra, which are both well-known from the theory of network flows. \Textcite{Rockafellar:1984} gives an extensive account of the related theory, based on which we present the short overview given in this section.

Given a directed graph $G=(V,E)$ with incidence matrix $A$, the (uncapacitated) \emph{feasible flow} polyhedron (for some right-hand-side vector $d$ satisfying $\mathbbm{1}^\top d = 0$) can be written as
\begin{equation*}
	Q_F(G):= \Set*{x \in \reals^E  \given  Ax = d }.
\end{equation*}

Note that, while $Q_F(G)$ does not restrict the flow by any edge capacities nor sign constraints, we could add capacity constraints without changing the properties that we will discuss below, albeit at the cost of a somewhat more cumbersome notation.

Now, suppose that we want to treat positive and negative flows on an edge separately, say by assigning a (possibly different) positive cost $\fu_e\geq 0$ and $-\fl_e \geq 0$ to positive and negative flows, respectively.
To achieve this, we can split up the vector $x$ into a positive part $x^+$ and a negative part $x^-$ to obtain
\begin{equation}
	\min \Set*{{\fu}^\top x^+  - {\fl}^\top x^-  \given  x^+, x^- \geq 0, A(x^+-x^-)=d}.\label{eq:flows_differentials_poly}
\end{equation}
Observe that since $\fu_e, -\fl_e \geq 0$, we can assume \Wlog~that for each edge $e$, at most one of $x^+_e$ and $x^-_e$ is strictly positive.

The feasible flow polyhedron (written in the form of \cref{eq:flows_differentials_poly}) is famous for a very neat characterization of its combinatorial structure:
Any of its extremal solutions can be characterized by the set of components of $x^+$ and $x^-$ for which the sign constraints are binding. It is convenient now to differentiate between the set $\mathcal{N}$ of edges $e$ for which both $x^+_e$ and $x^-_e$ are fixed to $0$ and its complement $\mathcal{B}:=E\setminus \mathcal{N}$,
the set of edges $e$ for which one of $x^+_e$ and $x^-_e$ may assume a value strictly greater than $0$.

Assuming that $G$ is weakly connected, we have $\rank(A)=|V|-1$ by \cref{prop:ABprop}.
For the pair $(\mathcal{B},\mathcal{N})$ to uniquely describe an extremal solution, it is therefore necessary that $\mathcal{N}$ contains $|E|-(|V|-1)$ edges corresponding to linearly independent columns of $A$.
We call $\mathcal{B}$ a \emph{basis} if $\mathcal{B}$ consists of exactly $|V|-1$ such edges.

The following characterization now holds with respect to the bases of the above polyhedron: $\mathcal{B}$ is a basis if and only if the edges of $G$ which correspond to those columns of $A$ indexed by $\mathcal{B}$ form an undirected spanning tree in $G$. This can be seen as follows: The columns for any cycle in $G$ are always linearly dependent (as the column of any edge can be obtained as the sum of positive/negative columns of the other edges), hence the sets of $|V|-1$ linearly independent columns are exactly those corresponding to spanning trees.

A similar characterization holds for the optimization problem dual to \cref{eq:flows_differentials_poly}, the optimal differential problem. It asks for a potential $\phi$ maximizing the linear function $d^\top \phi$ over the polyhedron $Q_D(G)$ of \emph{feasible differentials} \cite{Rockafellar:1984} defined as follows:
\begin{equation*}
	Q_D(G):= \Set*{\phi \in \reals^V  \given  \fl \leq A^\top \phi \leq \fu }
\end{equation*}
The variables in this new polyhedron do not have to satisfy any (individual) bounds or sign restrictions.
An extreme point of this polyhedron is hence characterized by the set of components of the vector $A^\top \phi$ that are fixed to one of their respective capacity bounds. Dual to the notation above, we can call the set of edges that index these rows $\mathcal B$ and analogously $\mathcal N := E \setminus \mathcal B$. Again, by \cref{prop:ABprop}, $\rank(A)=|V|-1$ and hence in order to uniquely identify (up to translation along $\mathbbm{1}$) the vector $\phi$ conforming the specifications of $(\mathcal{B},\mathcal{N})$, we need that $\mathcal{B}$ indexes $|V|-1$ linearly independent rows of $A^\top$ (or equivalently columns of $A$). As argued above, a set $\mathcal{B}$ of exactly $|V|-1$ such rows satisfies this requirement if and only if the corresponding edges of $G$ form a spanning tree. Furthermore, the potential $\phi$ uniquely determined by $(\mathcal{B},\mathcal{N})$ (up to translation along $\mathbbm{1}$) is feasible if and only if $\fl_{\mathcal N} \leq A_{\mathcal N}^\top \phi \leq \fu_{\mathcal N}$.

Note that the roles of $\mathcal B$ and $\mathcal N$ in both cases are reversed: While for the feasible flow polyhedron, the entries of $x$ indexed by $\mathcal B$ are the ones that are free to take values between the respective bounds, the inequalities indexed by $\mathcal B$ in the feasible differential polyhedron are the ones that are forced to be binding. A summary of the characterizations outlined above is provided in \cref{tab:flows_and_differentials}.

\begin{table}
\scriptsize\centering
\begin{tabular}{l|c|c}
					&	feasible flows	&feasible differentials \\ \hline
	feasible set	&	$Q_F(G)$		&	$Q_D(G)$			\\
	solution conforms to $(\mathcal{B},\mathcal{N})$ if\dots & $x_{\mathcal N}:=x^+_{\mathcal N}-x^-_{\mathcal N}=0$ & all entries of $A_{\mathcal B}^\top\phi$ at upper/lower bound \\
	conforming solution is extremal if\dots & \multicolumn{2}{c}{$\mathcal{B}$ indexes a spanning tree}\\
		or (equivalently) if\dots& \multicolumn{2}{c}{$\mathcal{N}$ is a maximal set of edges not intersecting every spanning tree}\\
\end{tabular}
\caption{Characterizing properties of (extremal) solutions for the polyhedra of feasible flows and feasible differentials.
}
\label{tab:flows_and_differentials}
\end{table}

The connections of the above theory of basic flows and differentials to the polyhedron $\QDCf(G)$ defined in \cref{eq:q_f} above are easy to make:
The feasible flow polyhedron $Q_F(G)$ (possibly with added edge capacities) can be seen as a variant of the polyhedron $\QDCf(G)$ with $\pl = \pu = -d$ and without the differential constraint $B^{\top}\phi=f$.
On the other hand, $Q_D(G)$ is (up to a full-rank linear transformation) a special case of the polyhedron $\QDCf(G)$ with $\pmax = (-\infty,\infty)$.

In this sense, \emph{differential flows} jointly generalize the notions of flows and differentials. This also explains our choice of terminology (see \cref{def:feas_diff_flow}): A differential flow combines the defining properties of both flows and differentials, placing it somehow \emph{in between} the two concepts.
This raises the following question, which we answer in the following: Can a combinatorial characterization of extremal points along the lines of \cref{tab:flows_and_differentials} be recovered for $\QDCf(G)$, given that $\QDCf(G)$ imposes constraints of both types, $B^{\top}\phi=f$ and $\pl \leq -Af \leq \pu$?

\section{$\alpha$-Forests and $\alpha$-Trees}

In the following we prove that, in almost all cases, a combinatorial characterization of the extremal points of $\QDCf(G)$ is indeed possible: Using a suitable notion of acyclicity, extremal points of $\QDCf(G)$ can be associated
with maximal acyclic collections of edges \emph{and vertices} which, again, correspond to the set of active inequalities. Furthermore, every such collection also corresponds to an extremal point of $\QDCf(G)$ in all but some exceptional cases.

To motivate the following definition, observe the following (non-exhaustive) list of sufficient criteria:

The point $f$ is extremal in $\QDCf(G)$ if
\begin{itemize}
	\item the edge capacity constraints are active for every edge in an undirected spanning tree (this is the case of the original feasible differential problem (see \cref{tab:flows_and_differentials}),
	\item the edge capacity constraints are active in all edges from a spanning tree \emph{except one}, and additionally the relaxed flow conservation constraint is binding in one of the endpoints of that missing edge,
	\item the relaxed flow conservation constraints are active in exactly $|V|-1$ vertices (this follows from \cref{prop:ABprop}),
	\item an edge capacity constraint is active on an edge $(v,w)$ and any $|V|-2$ relaxed flow conservation constraints are active that do not contain both $v$ and $w$.
\end{itemize}

These observations (which will later be proved formally in a more general context) motivate the following definition of a structure that can be used to characterize extremal solutions (see~\cref{fig:alpha-tree}).

\begin{definition}[$\alpha$-Forest]\label{def:dc_ext_forest}
	Let $G=(V,E,b,\pmax,\fmax)$ be a network. A pair $F=(E_F,V_F)$ consisting of a set of edges $E_F \subseteq E$ and a set of vertices $V_F \subseteq V$ in $G$ is called an \emph{$\alpha$-forest} in $G$ if there exists an injective function $\alpha_F: V_F \rightarrow  E\setminus E_F$, mapping each vertex from $V_F$ to a neighboring edge, which is not already in $E_F$, such that the set $E_F \cup \alpha_F(V_F)$ does not contain an undirected cycle. We call any such $\alpha_F$ a \emph{vertex-orientation map} for $F$ and any edge $e \in E$ \emph{active} in $F$ if $e \in E_F$. Analogously, for a vertex $v \in V$, we say that $v$ is \emph{active} in $F$ if $v \in V_F$.
	The \emph{size} of an $\alpha$-forest $F$ is defined by $|F| := |E_F|+ |V_F|$.
	If there is no $\alpha$-forest $F'$ in $G$ with $|F'| > |F|$, then we say that $F$ is \emph{maximal}.
\end{definition}

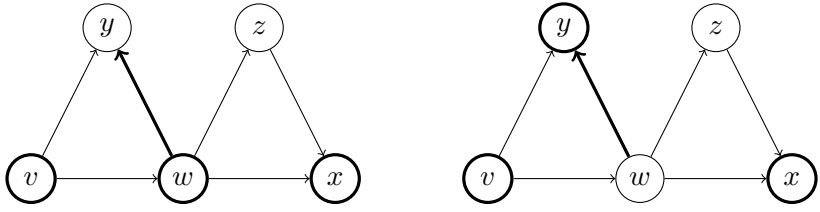
\begin{figure}
	\centering
	\begin{subfigure}{0.4\textwidth}
	\begin{tikzpicture}
		\node[gridnode, very thick] (v) at (0,0) {$v$};
		\node[gridnode, very thick] (w) at (2,0) {$w$};
		\node[gridnode, very thick] (x) at (4,0) {$x$};
		\node[gridnode] (y) at (1,2) {$y$};
		\node[gridnode] (z) at (3,2) {$z$};

		\draw (v)	edge[->] (w)
					edge[->] (y);
		\draw (w)	edge[->, very thick] (y)
					edge[->] (z)
					edge[->] (x);
		\draw (z)	edge[->] (x);
	\end{tikzpicture}
	\end{subfigure}
	\begin{subfigure}{0.4\textwidth}
	\begin{tikzpicture}
		\node[gridnode, very thick] (v) at (0,0) {$v$};
		\node[gridnode] (w) at (2,0) {$w$};
		\node[gridnode, very thick] (x) at (4,0) {$x$};
		\node[gridnode, very thick] (y) at (1,2) {$y$};
		\node[gridnode] (z) at (3,2) {$z$};

		\draw (v)	edge[->] (w)
					edge[->] (y);
		\draw (w)	edge[->, very thick] (y)
					edge[->] (z)
					edge[->] (x);
		\draw (z)	edge[->] (x);
	\end{tikzpicture}
	\end{subfigure}

	\caption{We visualize $\alpha$-forests by marking the sets of active vertices and edges in bold. The structure on the left is an $\alpha$-forest, as we can choose, \eg, $\alpha_F(v)=(v,w), \alpha_F(w)=(w,x), \alpha_F(x)=(z,x)$. The structure on the right, in contrast, is not an $\alpha$-forest: Since $\alpha_F$ needs to be injective and cannot map to edges already in $E_F$, it follows that $\alpha_F(y)=(v,y)$ and $\alpha_F(v)=(v,w)$, which closes a cycle.}
\label{fig:alpha-tree}
\end{figure}

Note that $E_F$ can contain edges between vertices not in $V_F$. Hence the vertices and edges in an $\alpha$-forest need not form a graph. Furthermore, for a given $\alpha$-forest $F$, the vertex-orientation map $\alpha_F$ is not necessarily unique.

Observe that any non-maximal $\alpha$-forest $F$ can be augmented to a maximal $\alpha$-forest: Let $\alpha_F$ be a corresponding vertex-orientation map. Since $F$ is non-maximal, there exists an $\alpha$-forest $F'$ with a corresponding vertex-orientation map $\alpha_F'$ such that $|E_{F'} \cup \alpha_{F'}(V_{F'})| > |E_F \cup \alpha_F(V_F)|$, where both sets do not contain an undirected cycle. From the augmentation property of the graphic matroid, we obtain that there is $e \in E_{F'} \cup \alpha_{F'}(V_{F'})$ such that $F^* := (E_F \cup \set{e}, V_F)$ is an $\alpha$-forest with $|F^*| > |F|$.

Finally, remember that we always assume that a network is weakly connected. Thus, we may call a maximal $\alpha$-forest an \emph{$\alpha$-tree} and, by the above observation, any $\alpha$-forest can be augmented to an $\alpha$-tree.

\begin{definition}[Conforming $\alpha$-forests]
	Given a feasible differential flow $f \in \QDCf(G)$, we say that the $\alpha$-forest $F$ \emph{conforms} with $f$ if
	\begin{itemize}
		\item $f_e \in \Set*{ \fl_e, \fu_e }$ for all $e \in E_F$
		\item $(-Af)_{v} \in \Set*{ \pl_{v}, \pu_{v} }$ for all $v \in V_F$.
	\end{itemize}

	Finally, $F$ is \emph{$f$-maximal} if there exists no $\alpha$-forest $F'$ conforming with $f$ such that $|F'| > |F|$.
\end{definition}

Remember that for the special case of $\pmax = (-\infty,\infty)$, we have already observed that $\QDCf(G)=Q_D(G)$ (with a suitable scaling of $\fl$ and $\fu$), \ie a feasible differential flow in that case is just a feasible differential. An $\alpha$-forest conforming with a flow $f$ is then simply a forest consisting of edges that are at their capacity limits. Moreover, we already know from \cref{sec:flows_and_differentials} that $f := A^\top\phi$ is extremal in $Q_D(G)$ if and only if there exists a spanning tree conforming with $f$ (note that there might be different trees characterizing the same solution).

For the general case, we start by formally defining the exact relation between $\alpha$-forests and points in
\begin{equation*}
	\QDCf(G) = \Set*{f \in \reals^E \given \exists \phi \in \reals^V:
	    \begin{pmatrix} 0 \\ p^- \\ f^- \end{pmatrix}  \le  \begin{pmatrix} B^{\top} & -I \\ 0 & -A \\ 0 & I \end{pmatrix} \begin{pmatrix}  \phi \\ f \end{pmatrix} \le \begin{pmatrix}  0 \\ p^+ \\ f^+ \end{pmatrix}}
\end{equation*}
that we would like to prove. Since the differential constraints are always active, a selection $\mathcal{B}$ of rows from the matrix
$\scolvec{-A\\I}$ uniquely determines $f \in \QDCf$ if and only if the matrix
	\begin{equation*}
		\colvec{B^\top & -I\\0 & \colvec{-A\\I}_{\mathcal{B}}}
	\end{equation*}
	has maximal rank, \ie if any vector $x$ satisfying
	\begin{equation*}
		\colvec{B^\top & -I\\0 & \colvec{-A\\I}_{\mathcal{B}}}x = 0
	\end{equation*}
	is of the form $x=(\phi,0)$ with $\phi := \lambda \cdot \mathbbm{1} \in \reals^V$ (\cf \cref{eq:all-ones}).

Another way to look at this is the following:
Recall from \cref{prop:ABprop} that $\rank(B)=|V|-1$ and hence we can alternatively consider the polyhedron $Q_\phi$ from \cref{eq:q_phi}. Then, an extremal point in $\QDCf(G)$ is uniquely determined by a subset of active constraints from
\begin{gather*}
	\fl \leq B^\top \phi \leq \fu \\ \pl \leq -AB^\top\phi \leq \pu
\end{gather*}
	such that the corresponding submatrix of $\scolvec{B^\top \\ AB^\top}$ has rank $|V|-1$.

  Before we delve into the question of when exactly an $\alpha$-tree conforming with a solution certifies that this solution is extremal, we first settle the inverse question:
  We prove that for every extremal solution $f \in \QDCf(G)$, we can find an $\alpha$-tree in $G$ that conforms with $f$ (\cref{thm:alpha-tree-ex}).
  In the following, let $N_G(v)$ denote the \emph{neighborhood} of $v$, i.e. $N_G(v):=\Set{w \in V \given (v,w) \in E \text{ or } (w,v) \in E}$.

\begin{lemma}\label{lem:at-matrix}
	Let $G=(V,E,b,\pmax,\fmax)$ be a network. Let $V^* := \Set{v_1,v_2,\dots,v_k} \subseteq V$ denote a set of $k$ vertices and let $E^* \subseteq E$ denote a set of $|V|-1-k$ edges such that
		\begin{equation}
			\rank \begin{pmatrix}
				(B^\top)_{E^*}\\
				(AB^\top)_{V^*}
			\end{pmatrix}=|V|-1.\label{eq:alpha_tree_rank}
		\end{equation}
	Then, $G^* := (V,E^*)$ has $k+1$ connected components and, denoting the corresponding vertex sets by $S_1, \dotsc, S_{k+1}$, there exists a $k \times (k+1)$-matrix $C = (\gamma_{ij})_{\substack{i \in \ints{k}\\j \in \ints{k+1}}}$ with
		\begin{equation*}
			\gamma_{ij} \begin{cases}
						> 0 & \text{if $v_i \in S_j$}\\
						< 0 & \text{if $v_i \notin S_j$ and $N_G(v_i) \cap S_j \neq \emptyset$}\\
						0	& \text{else.}
						\end{cases}
		\end{equation*}
    such that $\rank(C)=k$ and $C \cdot \mathbbm{1} = 0$.
\end{lemma}

\begin{proof}
	The rank condition \cref{eq:alpha_tree_rank} implies in particular that the rows of $B^\top$ which correspond to the edges in $E^*$ are linearly independent. Since $B$ is (up to linear scaling of columns) the incidence matrix of the graph $(V,E)$, this implies that the graph $(V,E^*)$ is acyclic and consists of $k+1$ connected components.
Let $M := \scolvec{	(B^\top)_{E^*}\\(AB^\top)_{V^*}}$.
	We assume \Wlog~that the columns of $M$ are ordered in such a way that those corresponding to vertices in $S_1$ come first, followed by the columns corresponding to vertices in $S_2$ and so on. We now scale every row of $(B^\top)_{E^*}$ in such a way that its two entries are $1$ and $-1$ (\eg~we divide the row corresponding to the edge $e$ by $b_e$).
	We obtain a matrix of the following structure:
	\begin{equation*}
		M' = \begin{pmatrix}
				(A^\top)_{S_1}	&0				& \dotsb	&0\\
				0				&(A^\top)_{S_2}	& \dotsb	&0\\
				\vdots			&\vdots			&\ddots	&\vdots\\
				0				&0				& \dotsb	&(A^\top)_{S_{k+1}}\\\hline
				\multicolumn{4}{c}{M^*}
			\end{pmatrix}.
	\end{equation*}
	Obviously, we have that $\rank(M)=\rank(M')$ and each of the blocks $(A^\top)_{S_j}$ is (the transpose of) the incidence matrix of the connected component of $G^*$ with vertex set $S_j$, which is a tree spanning the set $S_j$.
	
	The matrix $M^*$, on the other hand, is a selection of rows from the nodal admittance matrix $AB^\top$, with each row corresponding to a vertex $v_i \in V^*$. For every $v_i \in V^*, w \in V$, it therefore holds that
	\begin{equation*}
		M^*_{v_iw} = \begin{cases}
						\sum_{e \in \delta(v_i)} b_e & w=v_i\\
						-b_{v_iw} & (v_i,w) \in E\\
						-b_{wv_i} & (w,v_i) \in E\\
						0 & \text{else.}
					\end{cases}
	\end{equation*}
	In particular, $M^*_{v_iv_i} > 0$, $M^*_{v_iw} \leq 0$ for all $w \neq v_i$ and $\sum_{w \in V} M^*_{v_iw} = 0$.

	We now select for each connected component of $G^*$ a representative vertex $w_j \in  S_j$. Remember that $(A^\top)_{S_j}$ is the transpose of the incidence matrix of a spanning tree on the set $S_j$. Adding a multiple $\lambda$ of the row of $(A^\top)_{S_j}$ corresponding to the edge $(v,w) \in E$ to a row $m$ of $M^*$ thus subtracts $\lambda$ from the entry $M^*_{mv}$ and adds $\lambda$ to the entry $M^*_{mw}$. Moving along the edges in the spanning tree of $S_j$, we can hence add a suitable linear combination of rows of $(A^\top)_{S_j}$ to each row $m$ of $M^*$, eliminating all entries of $M^*$ from columns corresponding to vertices in $S_j\setminus{w_j}$ and adding the corresponding values to the entry $M^*_{mw_j}$. We obtain a matrix $M^{**}$ such that
	\begin{equation*}
		\rank(M) = \rank(M') = \rank\begin{pmatrix}
				(A^\top)_{S_1}	&0				& \dotsb	&0\\
				0				&(A^\top)_{S_2}	& \dotsb	&0\\
				\vdots			&\vdots			&\ddots	&\vdots\\
				0				&0				& \dotsb	&(A^\top)_{S_{k+1}}\\\hline
				\multicolumn{4}{c}{M^{**}}
			\end{pmatrix},
	\end{equation*}
	where $M^{**}$ is of the form
	\begin{equation*}
		\kbordermatrix{
					~ & w_1 & ~ &  & ~ & w_2 & ~ & ~ & ~ & \dotsb & ~ & ~ & ~ & w_{k+1} \\
					v_1 & \gamma_{11}	& 0		& \dotsb & 0		& \gamma_{12}	& 0		& \dotsb & 0		& \dotsb &	0		& \dotsb & 0		& \gamma_{1(k+1)}\\
					v_2 & \gamma_{21}	& 0		& \dotsb & 0		& \gamma_{22}	& 0		& \dotsb & 0		& \dotsb &	0		& \dotsb & 0		& \gamma_{2(k+1)}\\
					\vdots & \vdots 		&\vdots	& \ddots & \vdots	& \vdots 		&\vdots	& \ddots & \vdots	&		 &	\vdots	& \ddots & \vdots	& \vdots\\
					v_k & \gamma_{k1}	& 0		& \dotsb & 0		& \gamma_{k2}	& 0		& \dotsb & 0		& \dotsb &	0		& \dotsb & 0		& \gamma_{k(k+1)}\\
				}.
	\end{equation*}

	Since the row-sums of all rows of $M'$ were 0, the same holds true for $M^{**}$, \ie $\sum_{j \in \ints{p}} \gamma_{ij} = 0$ for all $i \in \ints{k}$. Furthermore, $\gamma_{ij} = \sum_{w \in S_j} M^*_{v_i,w} \leq 0$ for all $S_j$ that do not contain $v_i$ and $\gamma_{ij} = 0$ for all $S_j$ that do not contain a neighbor of $v_i$. This implies that, if $v_i \in S_{j^*}$, then $\gamma_{ij^*} > 0$: Otherwise, all entries of the row $M^{**}_{v_i}$ would be $0$, a contradiction with $M^{**}$ having full row rank. We can now drop all zero columns from $M^{**}$ to obtain the matrix $C$ with the desired properties.
\end{proof}

In preparation of Theorem \ref{thm:alpha-tree-ex}, the following lemma will serve to capture the relation between the different connected components of $G^*$ and active vertices within these components. It uses a bipartite graph $H := (W \cupdot S, R \cupdot U)$, where the set $W$ will be used to represent the set $V^*$ of \emph{active} vertices in a network while the set $S$ represents the set of components connected by \emph{active} edges in $E^*$ (the connected components of the graph $G^*$). An edge $\Set{v,s} \in R$ will mean that the vertex in $V^*$ which corresponds to $v$ lies in the connected component that corresponds to $s$. The set $U$ of edges represents possible connections by edges in $E \setminus E^*$ between vertices in $V^*$ and connected components of $G^*$ that the vertex does not belong to. A selection $U^*$ of edges in $U$ finally will represent the set of connections between connected components of $G^*$ actually made by active vertices through the vertex-orientiation map of an $\alpha$-forest (see \cref{fig:at-graph}).

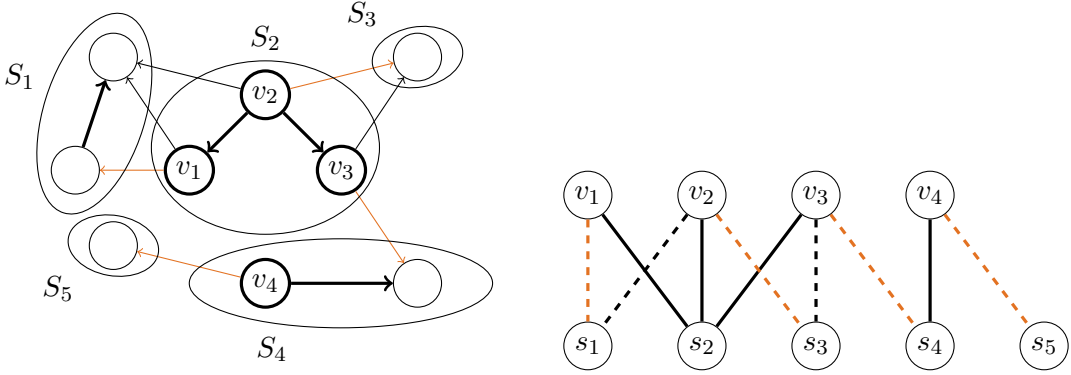
\begin{figure}
\centering
\begin{subfigure}{0.5\textwidth}
\begin{tikzpicture}
	\node[gridnode, very thick] (v1) at (0,0) {$v_1$};
	\node[gridnode, very thick] (v2) at (1,1) {$v_2$};
	\node[gridnode, very thick] (v3) at (2,0) {$v_3$};
	\node[gridnode, very thick] (v4) at (1,-1.5) {$v_4$};

	\node[gridnode] (s11) at (-1.5,0) {};
	\node[gridnode] (s12) at (-1,1.5) {};
	\node[gridnode] (s3) at (3,1.5) {};
	\node[gridnode] (s4) at (3,-1.5) {};
	\node[gridnode] (s5) at (-1,-1) {};

	\draw[very thick]	(s11)	edge[->] (s12)
						(v2)	edge[->] (v1)
								edge[->] (v3)
						(v4)	edge[->] (s4);
	\draw				(v1)	edge[->, accentuating orange] (s11)
								edge[->] (s12)
						(v2)	edge[->] (s12)
								edge[->, accentuating orange] (s3)
						(v3)	edge[->] (s3)
								edge[->, accentuating orange] (s4)
						(v4)	edge[->, accentuating orange] (s5);

	\node[inner sep=-1mm, fit=(s11)(s12), draw, ellipse, rotate=-17, label={left:$S_1$}] {};
	\node[draw, ellipse, minimum width=12mm, minimum height=8mm, rotate=10, label={above left:$S_3$}] at (s3) {};
	\node[inner sep=0.8mm, fit=(v4)(s4), draw, ellipse, label={below left:$S_4$}] {};
	\node[draw, ellipse, minimum width=12mm, minimum height=8mm, rotate=-10, label={below left:$S_5$}] at (s5) {};
	\node[draw, ellipse, minimum width=30mm, minimum height=23mm, label=above:$S_2$] at (1,0.3) {};
\end{tikzpicture}
\end{subfigure}
\begin{subfigure}{0.45\textwidth}
\begin{tikzpicture}[xscale=1.5, yscale=2]
	\node[gridnode] (v1) at (0,0) {$v_1$};
	\node[gridnode] (v2) at (1,0) {$v_2$};
	\node[gridnode] (v3) at (2,0) {$v_3$};
	\node[gridnode] (v4) at (3,0) {$v_4$};

	\node[gridnode] (s1) at (0,-1) {$s_1$};
	\node[gridnode] (s2) at (1,-1) {$s_2$};
	\node[gridnode] (s3) at (2,-1) {$s_3$};
	\node[gridnode] (s4) at (3,-1) {$s_4$};
	\node[gridnode] (s5) at (4,-1) {$s_5$};

	\draw[very thick]	(s2)	edge (v1)
								edge (v2)
								edge (v3)
						(s4)	edge (v4);
	\draw[very thick, dashed]	(v1)	edge[accentuating orange] (s1)
								(v2)	edge (s1)
										edge[accentuating orange] (s3)
								(v3)	edge (s3)
										edge[accentuating orange] (s4)
								(v4)	edge[accentuating orange] (s5);
\end{tikzpicture}
\end{subfigure}
\caption{On the left, a network with the sets $V^*$ and $E^*$ shown in bold, as well as the partition of vertices induced by the connected components of the graph $G^* = (V, E^*)$. On the right, the corresponding bipartite graph $H := (W \cupdot S, R \cupdot U)$ satisfying the conditions from \cref{lem:at-graph}. Solid edges are those from the set $R$, dashed edges are those from the set $U$. In orange, a selection of edges $U^*$ guaranteed by \cref{lem:at-graph} is shown. On the left, a possible translation into an $\alpha$-forest's vertex-orientation map $\alpha$ (see \cref{thm:alpha-tree-ex}) is indicated by the orange edges ($\alpha$ maps every active vertex to the neighboring orange edge).}
\label{fig:at-graph}
\end{figure}

Similarly to our definitions of $\delta^\inn(v)$ and $\delta^\outn(v)$, we denote in the case of an undirected graph by $\delta(v)$ the set of edges incident with $v$. Furthermore, we define the neighborhood of a set of vertices by $N_G(W):=\Set{v \in V \setminus W \given \exists w \in W: \Set{v,w} \in E}$.

\begin{lemma}\label{lem:at-graph}
	Let $H := (W \cupdot S, R \cupdot U)$ be a bipartite graph with $|S| = |W| + 1$ satisfying the following conditions:
	\begin{enumerate}
		\item\label{it:at-degree} $\delta(v) \cap R = 1$ for all $v \in W$
		\item\label{it:at-subset} $|N_H(W')| \geq |W'| + 1$ for every $W' \subseteq W$.
	\end{enumerate}
	Then, there exists a selection of edges $U^* \subseteq U$ such that
	\begin{enumerate}[label=\roman*)]
		\item\label{it:at-single} $\delta(v) \cap U^* = 1$ for all $v \in W$ and
		\item\label{it:at-connected} the graph $(W \cupdot S, R \cupdot U^*)$ is connected (and therefore a tree).
	\end{enumerate}
\end{lemma}

\begin{proof}
	We first observe that \cref{it:at-subset}, together with \cref{it:at-degree}, implies in particular that $\delta(v) \cap U \geq 1$ for all $v \in W$. Thus, $|U| \geq |W|$ and since $H$ is bipartite, a selection $U^* \subseteq U$ of edges with $\delta(v) \cap U^* = 1$ always exists.

	We prove the statement by induction over $|U|$. If $|U| = |W|$, then every vertex $v \in W$ is incident with exactly one edge from $U$. The selection $U^* := U$ thus satisfies \cref{it:at-single}. If we suppose that \cref{it:at-connected} does not hold, \ie~$(W \cupdot S, R \cupdot U)$ is not connected, then we can choose a connected component $C$ and denote by $W_C$ the subset of vertices from $W$ that is covered by $C$. Since $C$ is a connected component, $N_H(W_C) \cap N_H(W \setminus W_C) = \emptyset$ and by \cref{it:at-subset} we obtain $|N_H(W_C)| > |W_C|$. Hence,
	\begin{equation*}
		|N_H(W \setminus W_C)| \leq |S \setminus N_H(W_C)| = |S| - |N_H(W_C)| < |W| + 1 - |W_C| = |W \setminus W_C| + 1,
	\end{equation*}
	a contradiction to \cref{it:at-subset}.

	Now, let us assume that $|U| > |W|$, which means that there exists $v \in W$ with $\delta(v) \cap U \geq 2$. Choose two distinct edges $e_1,e_2 \in \delta(v) \cap U$ and denote the corresponding neighbors of $v$ in $S$ by $s_1$ and $s_2$, respectively.

	We claim that at least one of the two graphs $H_1 := (W \cupdot S, R \cupdot U \setminus \Set{e_1})$ and $H_2 := (W \cupdot S, R \cupdot U \setminus \set{e_2})$ must satisfy the condition \cref{it:at-subset}. By induction, we then obtain a selection $U^*$ of edges which satisfies \cref{it:at-single,it:at-connected} with respect to the graph $H_1$ (or $H_2$) and thus also for $H$, proving the statement.

	For a contradiction, suppose otherwise. Then, there exists a set $W_1 \subset W$ which violates \cref{it:at-subset} in the graph $H_1$ and a set $W_2 \subset W$ which violates \cref{it:at-subset} within $H_2$. Since neither $W_1$ nor $W_2$ violate \cref{it:at-subset} within $H$, it must hold that $v \in W_1 \cap W_2$.

	Now, let $W'_1 := W_1 \setminus W_2$, $W'_2 := W_2 \setminus W_1$, and $W_{12} := W_1 \cap W_2 \setminus \Set{v}$. We use the following statements, which we successively derive from each other below:
	\begin{enumerate}[label=\Roman*)]
		\item\label{it:atg-h1v1} $|N_{H_1}(W_1)| = |W_1| = |N_H(W_1)| - 1$ and\\
				$|N_{H_2}(W_2)| = |W_2| = |N_H(W_2)| - 1$
		\item\label{it:atg-v1*} $|N_H(W'_1) \cap N_H(W_{12})| \geq |N_H(W'_1)| + |N_H(W_{12})| - |W'_1| - |W_{12}| - 1$ and\\
				$|N_H(W'_2) \cap N_H(W_{12})| \geq |N_H(W'_2)| + |N_H(W_{12})| - |W'_2| - |W_{12}| - 1$
		\item\label{it:atg-v12} $|N_H(W'_1 \cupdot W_{12} \cupdot W'_2)| \leq |W'_1| + |W_{12}| + |W'_2| + 1$
		\item\label{it:atg-nv1} $N_H(v) \setminus \Set{s_1} \subset N_H(W_1 \setminus \Set{v})$ and
				$N_H(v) \setminus \Set{s_2} \subset N_H(W_2 \setminus \Set{v})$
	\end{enumerate}
	Statements \cref{it:atg-v12} and \cref{it:atg-nv1} together will then prove that the original graph $H$ would already violate \cref{it:at-subset}, a contradiction.

	To prove \cref{it:atg-h1v1} and \cref{it:atg-v1*}, let $i \in \Set{1,2}$. From the assumption that $W_i$ violates \cref{it:at-subset} in the graph $H_i$, but satisfies \cref{it:at-subset} in the graph $H$, we can immediately follow that $|N_{H_i}(W_i)| \leq |W_i| \leq |N_H(W_i)| - 1$. On the other hand, $H_i$ lacks only a single edge compared with $H$ and thus $|N_{H_i}(W_i)| \geq |N_H(W_i)|-1$. This proves \cref{it:atg-h1v1}.

	We can now conclude that
	\begin{align*}
		|W'_i| + |W_{12}| + 1
		&= |W_i| \overset{\cref{it:atg-h1v1}}{\geq} |N_{H_i}(W_i)| \geq |N_{H_i}(W'_i \cupdot W_{12})| = |N_H(W'_i \cupdot W_{12})|\\
			&= |N_H(W'_i)| + |N_H(W_{12})| - |N_H(W'_i) \cap N_H(W_{12})|,
	\end{align*}
	which proves \cref{it:atg-v1*}.

	Using the inclusion/exclusion principle, it follows
	\begin{align*}
		|N_H(W'_1 \cupdot W_{12} \cupdot W'_2)|
		&= |N_H(W'_1)| + |N_H(W_{12})| + |N_H(W'_2)|\\
			&\qquad - |N_H(W'_1) \cap N_H(W_{12})|- |N_H(W'_2) \cap N_H(W_{12})| \\
			&\qquad \underbrace{- |N_H(W'_1) \cap N_H(W'_2)| + |N_H(W'_1) \cap N_H(W'_2) \cap N_H(W_{12})|}_{\leq 0}\\
		&\overset{\mathclap{\cref{it:atg-v1*}}}{\leq}|N_H(W'_1)| + |N_H(W_{12})| + |N_H(W'_2)|\\
			&\qquad - (|N_H(W'_1)| + |N_H(W_{12})| - |W'_1| - |W_{12}| - 1) \\
			&\qquad - (|N_H(W'_2)| + |N_H(W_{12})| - |W'_2| - |W_{12}| - 1)\\
		&= - |N_H(W_{12})| + |W'_1| + |W_{12}| + |W'_2| + |W_{12}| + 2\\
		&\overset{\cref{it:at-subset}}{\leq}
		|W'_1| + |W_{12}| + |W'_2| + 1,
	\end{align*}
	which proves \cref{it:atg-v12}.

  To prove \cref{it:atg-nv1} let again $i \in \Set{1,2}$ and observe that $s_i \notin N_H(W_i \setminus \Set{v})$, since otherwise $|N_{H_i}(W_i)| = |N_H(W_i)|$, contradicting \cref{it:atg-h1v1}. Furthermore,
	\begin{align*}
		|N_H(W_i) \setminus N_H(W_i \setminus \Set{v})| &= |N_H(W_i)| - |N_H(W_i \setminus \Set{v})|\\
			&\overset{\mathclap{\cref{it:atg-h1v1},\cref{it:at-subset}}}{\leq}\, (|N_{H_i}(W_i)| + 1) - (|W_i \setminus \Set{v}| + 1)\\
			&\overset{\cref{it:atg-h1v1}}{\leq} |W_i| + 1 - (|W_i| - 1 + 1) = 1,
	\end{align*}
	which shows $N_H(W_i) \setminus N_H(W_i \setminus \Set{v}) = \Set{s_i}$ and hence $N_H(v) \setminus \Set{s_i} \subset N_H(W_i \setminus \Set{v})$,
	which proves \cref{it:atg-nv1}.

	In particular, since $s_1,s_2 \in N_H(v)$, this means that $s_2 \in N_H(W_1 \setminus \Set{v})$ and $s_1 \in N_H(W_2 \setminus \Set{v})$.
	But this implies that
	\begin{equation*}
		N_H(v) \overset{\cref{it:atg-nv1}}{\subset} N_H(W_1 \setminus \Set{v}) \cup N_H(W_2 \setminus \Set{v}) = N_H((W_1 \cup W_2) \setminus\Set{v}) = N_H(W'_1 \cupdot W_{12} \cupdot W'_2),
	\end{equation*}
	which results into
	\begin{align*}
		|N_H(W'_1 \cupdot W_{12} \cupdot W'_2 \cupdot \Set{v})| &= |N_H(W'_1 \cupdot W_{12} \cupdot W'_2)|\\
			&\overset{\mathclap{\cref{it:atg-v12}}}{=}\, |W'_1| + |W_{12}| + |W'_2| + 1 = |W'_1 \cupdot W_{12} \cupdot W'_2 \cupdot \Set{v}|,
	\end{align*}
	a contradiction with \cref{it:at-subset}.

	It follows that our assumption was false and that indeed, at least one of the graphs $H_1$ or $H_2$ satisfies \cref{it:at-subset}. As the set $R$ remains unchanged that graph obviously satisfies \cref{it:at-degree}, too. Thus the statement follows by induction.
\end{proof}

\begin{theorem}\label{thm:alpha-tree-ex}
	Let $G=(V,E,b,\pmax,\fmax)$ be a network and $f \in \QDCf(G)$. If $f$ is extremal, then there exists an $\alpha$-tree in $G$ that conforms with $f$.
\end{theorem}

\begin{proof}
	Let $f \in \QDCf(G)$ be extremal, then there exist at least $|V|-1$ linearly independent active inequalities. From \cref{prop:ABprop}, we obtain that, equivalently, there exists a vector $\phi \in \reals^V$ and $|V|-1$ linearly independent inequalities in \cref{eq:q_phi} which are active in $\phi$. Let $E^*$ and $V^* := \Set{v_1,v_2,\dotsc,v_k}$ be the set of edges and vertices, respectively, corresponding to these constraints and
  $	M := \scolvec{	(B^\top)_{E^*}\\ (AB^\top)_{V^*}} \in \R^{(|V|-1) \times |V|}$. Then
	$M$ has $\rank(M)=|V|-1$, thus satisfying the conditions of \cref{lem:at-matrix}.
	Hence, the graph $(V,E^*)$ consists of $k+1$ (weakly) connected components with vertex sets $S_1, S_2, \dotsc, S_{k+1} \subset V$ and there exists a $k \times (k+1)$-dimensional matrix $C = (\gamma_{ij})_{\substack{i \in \ints{k}\\j \in \ints{p}}}$ with $\rank(C)=k$ such that $C \cdot \mathbbm{1} = 0$ and
	\begin{equation*}
		\gamma_{ij} \begin{cases}
					> 0 & \text{if $v_i \in S_j$}\\
					< 0 & \text{if $v_i \notin S_j$ and $N_{(V,E^*)}(v_i) \cap S_j \neq \emptyset$}\\
					0	& \text{else.}
					\end{cases}
	\end{equation*}

	Now, let $S := \Set{S_1,S_2,\dotsc,S_{k+1}}$, $R := \Set{\Set{v_i,s_j} \given \gamma_{ij} > 0}$, and $U := \Set{\Set{v_i,s_j} \given \gamma_{ij} < 0}$. For the graph $H := (V^* \cupdot S, R \cupdot U)$, the following holds:
	\begin{itemize}
		\item $H$ is bipartite with $|S| = |V^*| + 1$ and
		\item $\delta(v) \cap R = 1$ for all $v \in V^*$, since every $v$ belongs to exactly one connected component $S_j$, $j \in [k+1]$.
	\end{itemize}
	Furthermore, we claim that for every $V' \subseteq V^*$, it holds that $|N_H(V')| \geq |V'| + 1$. To see this, suppose otherwise. Let $\bar V \subset V^*$ denote a set such that $|N_H(\bar V)| \leq |\bar V|$. Reorder the rows and columns of $C$ in such a way that the first rows are those corresponding to vertices in $\bar V$ and the first columns are those corresponding to vertices in $N_H(\bar V)$. Let $V^{**}$ be a subset of $\bar V$ of size $|N_H(\bar V)|$ and let $C'$ be the $|V^{**}| \times |V^{**}|$-dimensional submatrix composed of the rows corresponding to vertices in $V^{**}$ and the columns corresponding to vertices in $N_H(\bar V)$. Then, $C$ can be written as
$C =  \begin{pmatrix}	C'& 0 \\	* & *	\end{pmatrix}$.
Since $C \cdot \mathbbm{1} = 0$, it follows that $C' \cdot \mathbbm{1} = 0$ and hence $\rank(C') < |V^{**}|$ which implies that $\rank(C)<k$, a contradiction. This proves that, indeed, $|N_H(V')| \geq |V'| + 1$ for every $V' \subseteq V$.

	The graph $H$ hence satisfies all the requirements of \cref{lem:at-graph} and we obtain a selection of edges $U^* \subseteq U$ with $\delta(v) \cap U^* = 1$ for every $v \in V^*$ such that the graph $H^* := (V^* \cupdot S, R \cupdot U^*)$ is connected. Since $|S| = |V|+1$ and $|R| = |U^*| = |V|$, this implies that $H^*$ is acyclic. For each $v_i \in V^*$, choose $j_i$ such that $\Set{v_i,s_{j_i}} \in U^*$ and choose $e_i \in E$ incident with $v_i$ such that $S_{j_i} \cap e_i \neq \emptyset$. Such an edge exists, since $\Set{v_i,s_{j_i}} \in U$ implies that $\gamma_{ij_i} < 0$ which in turn implies that $N_{(V,E^*)}(v_i) \cap S_{j_i} \neq \emptyset$.

	Now, for all $v_i \in V^*$, let $\alpha(v_i) := e_i$. Since otherwise $H^*$ would contain an undirected cycle, we see that $\alpha$ is injective and $E^* \cup \alpha(V^*)$ acyclic. Hence, $F := (E^*,V^*)$ is an $\alpha$-forest, which can be seen using the (injective) vertex-orientation map $\alpha_F := \alpha$. Furthermore, $F$ is of size $|E^*| + |V^*| = |V|-1$, which means that it is in fact an $\alpha$-tree.
\end{proof}

Note that we see that the choice of $\alpha_F$ for an $\alpha$-tree $F$ is not necessarily unique.

\section{Non-degenerate Networks and Cactus Graphs}
\label{sec:classes_of_nondeg_graphs}

We have seen above that for every extremal solution $f \in \QDCf(G)$, one can find an $\alpha$-tree in $G$ that conforms with $f$. The reverse, however, is not necessarily the case, as the following example shows:

	\begin{figure}
	\centering
		\begin{tikzpicture}
		\node[gridnode, very thick] (1) at (0,0) {$w$};
		\node[gridnode, very thick] (2) at (0,4) {$v$};
		\node[gridnode] (3) at (-2,2) {$s$};
		\node[gridnode] (4) at (2,2) {$t$};

		\draw[very thick, ->] (1) edge (2);
		\draw[->] (1) edge (3);
		\draw[->] (2) edge (3);
		\draw[->] (1) edge (4);
		\draw[->] (2) edge (4);
		\end{tikzpicture}
	\caption{If $\frac{b_{wt}}{b_{ws}}=\frac{b_{vt}}{b_{vs}}$, then the $\alpha$-tree indicated by the active vertices and edges marked in bold conforms to a solution which is not extremal and hence the graph is $\alpha$-tree degenerate (see \cref{def:DC_nondeg_extremal_rep}).}
	\label{fig:DC-degenerate}
	\end{figure}
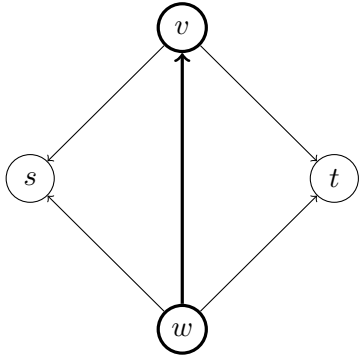

\begin{example}\label{ex:DC_nondeg_fails}
	Consider the network $G=(V,E,b,\pmax,\fmax)$ with $(V,E)$ as given in \cref{fig:DC-degenerate}. Let $f \in \QDCf(G)$ be such that the thick edges and vertices are those for which a corresponding constraint is binding. Note that there exists an $\alpha$-tree $F$ which conforms to $f$ (for example, $F=(\Set{ (v,w)}, \Set{v,w})$ with $\alpha_F(v)=(v,s)$ and $\alpha_F(w)=(w,t)$). However, $f$ need not be extremal in $\QDCf(G)$. More precisely, in the above case $f$ is extremal if and only if $\frac{b_{wt}}{b_{ws}}\neq\frac{b_{vt}}{b_{vs}}$.
\end{example}

In this section, we will derive a characterization of graphs for which we can guarantee that every solution for which a conforming $\alpha$-tree exists, is extremal.

\begin{definition}\label{def:DC_nondeg_extremal_rep}
	We say that a weighted graph $(V,E,b)$ is \emph{$\alpha$-tree non-degenerate} if
	  the following are equivalent for all networks $G=(V,E,b,\pmax,\fmax)$ (independently of the values $\pmax$ and $\fmax$):
	  \begin{enumerate}[label=\roman*)]
	  \item a point $f \in \QDCf(G)$ is extremal (i.e. the set of active edge and vertex constraints has rank $|V|-1$),
	  \item there exists an $\alpha$-tree in $G$ that conforms with $f$.
	  \end{enumerate}
	Otherwise, $(V,E,b)$ is \emph{$\alpha$-tree degenerate}.

	We say that a network $G=(V,E,b,\pmax,\fmax)$, is \emph{$\alpha$-tree (non-)degenerate} if the underlying weighted graph $(V,E,b)$ is \emph{$\alpha$-tree (non-)degenerate}.
\end{definition}

In order to avoid unnecessarily complicated language, we will simply call a network \emph{degenerate} (or \emph{non-degenerate}) if this is unambiguous. In a non-degenerate network, we can thus identify every extremal solution with a (not necessarily unique) selection of edges and vertices that form an $\alpha$-tree. Conversely, every solution for which such a selection of edges and vertices exists must be extremal.

Note that non-degeneracy of a network implies a more general relation between (not necessarily maximal) $\alpha$-forests and points in higher-dimensional faces:

\begin{remark}\label{lem:rank_min_tree_size}
Let $G=(V,E,b,\fmax,\pmax)$ be a non-degenerate network.
Then, the following are equivalent:
	\begin{enumerate}
		\item the matrix of active constraints in any point $f \in \QDCf(G)$ has rank $k$,
		\item there exists an $\alpha$-forest $F$ of size $k$ in $G$ that conforms with $f$.
	\end{enumerate}
\end{remark}

In the following we will present a characterization of all graphs on which any network is non-degenerate, irrespective of the given \tension vector $b$.

	The idea for the assumptions used in the upcoming main \cref{thm:chord-free_nondeg} are inspired by \cite{Zhang:2013}. However the statements of the following theorems of this subsection, as well as the proof techniques used, are entirely different. The idea of decomposing a network into subgraphs connected by a single vertex, that we will follow to prove \cref{thm:chord-free_nondeg}, was independently used in \cite{Leibfried:2015} to answer a question about the effect of partial relaxation of the DC equations in power flow models (which corresponds to relaxing the differential constraints in $\QDCf(G)$).

	\begin{lemma}\label{thm:1-con-nondeg}
		Let $H_1 := (V_1 \cup \Set{v^*},E_1)$ and $H_2 := (V_2 \cup \set{v^*},E_2)$ be two directed, weakly connected graphs with $V_1 \cap V_2 = \emptyset$ such that every network on these graphs is non-degenerate. Let $(V,E)$ be the graph
    with $V=V_1 \cup V_2 \cup \Set{v^*}$ and $E = E_1 \cup E_2$. Then, every network on $(V,E)$ is non-degenerate.
	\end{lemma}

	\begin{proof}
		Let $G=(V,E,b,\pmax,\fmax)$ be a network on $(V,E)$. Due to \cref{thm:alpha-tree-ex}, it suffices to show that if there exists an $\alpha$-tree that conforms with $f \in \QDCf(G)$, then $f$ is extremal.

		Thus, let $f \in \QDCf(G)$ and $F$ be an $\alpha$-tree that conforms with $f$ with vertex-orientation map $\alpha_F: V_F \rightarrow E \setminus E_F$.
		If $v^*$ is active in $F$, then we assume \Wlog~that $\alpha_F(v^*) \in E_2$.

		Let $\hat f,\tilde f \in \QDCf(G)$
    and $\lambda \in (0,1)$ such that $f = \lambda \hat f + (1-\lambda) \tilde f$. We will prove that this implies
    $\tilde f = \hat f = f$ and therefore $f$ must be extremal.

		We write $f^i$ for the restrictions of $f$ to
    $E_i$ and analogously $\hat f^i$ for $\hat f$ and $\tilde f^i$ for $\tilde f$, $i=1,2$. Consider the network $G_1$ obtained from restricting $G$ to the graph $H_1$ and changing the bounds for the vertex $v^*$ to $(-\infty, \infty)$ such that the constraint is never binding. This will make sure that the vertex $v^*$ can never be active in any $\alpha$-tree in the graph $H_1$.
		Moreover, let $F_1 := (E_{F_1}, V_{F_1})$ with $E_{F_1} := E_F \cap E_1$ and $V_{F_1} := E_F \cap E_1$. Then $F_1$ is an $\alpha$-forest with vertex-orientation map $\alpha_{F_1} := {\alpha_F|}_{V_1}$ which conforms with $f^1$. Since $F$ was an $\alpha$-tree in $G$ and either $v^* \notin V_F$ or, by our assumption, $\alpha_F(v^*) \in E_2$, it holds that $E_{F_1} \cup \alpha_{F_1}(V_{F_1}) = (E_F \cup \alpha_F(V_F)) \cap E_1$, which means that $F_1$ is actually an $\alpha$-tree in $G_1$.

		Furthermore, since $G_1$ is a network on $H_1$ it is non-degenerate and hence $f^1$ is an extremal flow in $G_1$.

    Finally, since $\hat f^1$ and $\tilde f^1$ are also feasible for $G_1$ and $f = \lambda \hat f + (1-\lambda) \tilde f$, we have in particular that $\hat f^1=\tilde f^1=f^1$.

		If $v^*$ is not active in $F$, then we can apply the same argument to conclude that $\hat f^2=\tilde f^2=f^2$ and thus $\hat f=\tilde f=f$.

    Hence we may assume that $v^*$ is indeed active in $F$, which, by our assumption, implies that $\alpha_F(v^*) \in E_2$.

		Since $F$ conforms with $f$, the flow $f$ satisfies either the upper or the lower vertex constraint in $v^*$ with equality in the network $G$. Suppose that the upper bound $\pu_{v^*}$ is binding (the argument would be the same if the lower bound $\pl_{v^*}$ is binding instead).
		Then
		\begin{align}
			\sum_{e \in \delta^\outn(v^*) \cap E_2} \hat f_{e} - \sum_{e \in \delta^\inn(v^*) \cap E_2} \hat f_{e}
			&\leq \pu_{v^*} - \sum_{e \in \delta^\outn(v^*) \cap E_1} \hat f_{e} - \sum_{e \in \delta^\inn(v^*) \cap E_1} \hat f_{e}\label{eq:1-con-nondeg-1}\\
			&= \pu_{v^*} - \sum_{e \in \delta^\outn(v^*) \cap E_1} f_{e} - \sum_{e \in \delta^\inn(v^*) \cap E_1} f_{e}\label{eq:1-con-nondeg-2}\\
			&= \sum_{e \in \delta^\outn(v^*) \cap E_2} f_{e} - \sum_{e \in \delta^\inn(v^*) \cap E_2} f_{e},\label{eq:1-con-nondeg-3}
		\end{align}
		where \cref{eq:1-con-nondeg-1} follows from the feasibility of $\hat f$ in $G$, \cref{eq:1-con-nondeg-2} follows from $f^1=\hat f^1$ and \cref{eq:1-con-nondeg-3} follows from the fact that $f$ satisfies the vertex constraint in $v^*$ with equality. By the same argument, $\sum_{e \in \delta^\outn(v^*) \cap E_2} \tilde f_{e} - \sum_{e \in \delta^\inn(v^*) \cap E_2} \tilde f_{e}
			\leq \sum_{e \in \delta^\outn(v^*) \cap E_2} f_{e} - \sum_{e \in \delta^\inn(v^*) \cap E_2} f_{e}$ which, since $f = \lambda \hat f + (1-\lambda) \tilde f$, implies
			\begin{align*}
				\sum_{e \in \delta^\outn(v^*) \cap E_2} f_{e} - \sum_{e \in \delta^\inn(v^*) \cap E_2} f_{e}
				&= \sum_{e \in \delta^\outn(v^*) \cap E_2} \hat f_{e} - \sum_{e \in \delta^\inn(v^*) \cap E_2} \hat f_{e}\\
				&= \sum_{e \in \delta^\outn(v^*) \cap E_2} \tilde f_{e} - \sum_{e \in \delta^\inn(v^*) \cap E_2} \tilde f_{e}.
			\end{align*}
		We now consider the network $G_2$ defined as follows: We restrict $G$ to the graph $H_2$ and fix both upper and lower bound of $v^*$ by
		\begin{equation*}
			\pl_{v^*} = \pu_{v^*} = \sum_{e \in \delta^\outn(v^*) \cap E_2} f_{e} - \sum_{e \in \delta^\inn(v^*) \cap E_2} f_{e}.
		\end{equation*}
		By our observations above, all of $f,\hat f,\tilde f$ are indeed feasible for $G_2$.
		Furthermore, the $\alpha$-forest $(E_F \cap E_2, \Set{v \in V_F \given \alpha_F(v) \in E_2}$ is an $\alpha$-tree in $G_2$ and, since $G_2$ is a non-degenerate network on $H_2$, we have that $\hat f^2=\tilde f^2=f^2$. In summary, we obtain $\hat f=\tilde f=f$.
	\end{proof}

	To be able to apply the above lemma, we now prove that networks on certain basic families of graphs are always non-degenerate. We start by considering networks of the extremely basic graph consisting of a single edge.

	\begin{lemma}\label{lem:K2-nondeg-rep}
		If $G=(V,E,b,\pmax,\fmax)$ is a network with $(V,E) = (\Set{ v,w },\Set{ (v,w) })$, then $G$ is non-degenerate.
	\end{lemma}

	\begin{proof}
		Let $f \in \QDCf(G)$ and let $F$ be an $f$-maximal $\alpha$-forest in $G$.
		If $f$ is not extremal then there exist $f',f'' \in \QDCf(G)$ with $f'_{vw} < f_{vw} < f''_{vw}$. This means that no vertex or edge constraint can be tight for $f$ and hence $F$ must be empty and cannot be an $\alpha$-tree.
	\end{proof}

	The next corollary follows immediately from repeated application of \cref{thm:1-con-nondeg} and \cref{lem:K2-nondeg-rep}, keeping in mind that every tree has a leaf.

	\begin{corollary}\label{cor:tree-nondeg-rep}
		If $G=(V,E,b,\pmax,\fmax)$ is a network such that the graph $(V,E)$ is a tree, then $G$ is non-degenerate.
	\end{corollary}

	Next, we consider networks for which the underlying graph is a cycle. This is a special case of a more general statement that will be proved later (\cref{thm:suff_cond_small_deg}). We provide a short direct proof here to keep this section self-contained.

	\begin{lemma}\label{lem:cycle-nondeg-rep}
		If $G=(V,E,b,\pmax,\fmax)$ is a network such that the graph $(V,E)$ is an (undirected) cycle, then $G$ is non-degenerate.
	\end{lemma}

\begin{proof}
	Let $f \in \QDCf(G)$ and $F$ be an $\alpha$-tree in $G$ that conforms with $f$. Since $|E| = |V|$, there exists $e^* = (s,t) \in E \setminus (E_F \cup \alpha_F(V_F))$. Now, let $b' = (b_{e})_{e \in E \setminus \set{e^*}}$, $\fmax' = (f_e)_{e \in E \setminus \set{e^*}}$. Furthermore, let $\pmax' \in \reals^V$ be given by $\pmax'_v = \pmax_v$ for all $v \in V \setminus \set{s,t}$, as well as $\pmax'_s = \pmax_s + f_{e^*}$ and $\pmax'_t = \pmax_t - f_{e^*}$ and consider the network $G'=(V,E \setminus \set{e^*},b',\pmax',\fmax')$. Then, $f' := (f_e)_{e \in E \setminus \set{e^*}} \in \QDCf(G')$ and $F$ is an $\alpha$-tree in $G'$ that conforms with $f'$. By \cref{cor:tree-nondeg-rep}, $G'$ is non-degenerate and hence $f'$ is extremal. Since $f$ satisfies the same edge and vertex constraints in $\QDCf(G)$, $f$ is also extremal.

	Conversely, let $f \in \QDCf(G)$ be extremal. Since $\QDCf(G)$ contains $|V|$ edge constraints, there exists $e^* \in E$ such that $\QDCf(G)$ contains a selection of edge and vertex constraints active in $f$, which does not include the edge constraint corresponding to $e^*$. Let the network $G'$ and $f'$ be defined as above. Then, $f'$ is extremal in $G'$ and by \cref{cor:tree-nondeg-rep}, there exists an $\alpha$-tree $F$ in $G'$ that conforms with $f'$. Using the same vertex-orientation map, $F$ is also an $\alpha$-tree in $G$. Furthermore, note that for every $e \in E \setminus \set{e^*}$ and $v \in V$ for which the corresponding edge/vertex constraint of $\QDCf(G')$ is active in $f'$, the corresponding edge/vertex constraint of $\QDCf(G)$ is also active in $f$. Hence, $F$ also conforms with $f$, which proves the statement.
\end{proof}

	The above results can be combined to characterize the family of graphs on which networks are guaranteed to be non-degenerate.

\begin{figure}
\centering
	\begin{tikzpicture}
	\node[gridnode] (1) at (0,0) {$w$};
	\node[gridnode] (2) at (0,4) {$v$};
	\node[gridnode] (3) at (-2,2) {$s$};
	\node[gridnode] (4) at (2,2) {$t$};

	\draw (1) edge (2);
	\draw (1) edge (3);
	\draw (2) edge (3);
	\draw (1) edge (4);
	\draw (2) edge (4);
	\end{tikzpicture}
\caption{The diamond graph (or Wheatstone bridge).}
\label{fig:wheatstone-graph}
\end{figure}
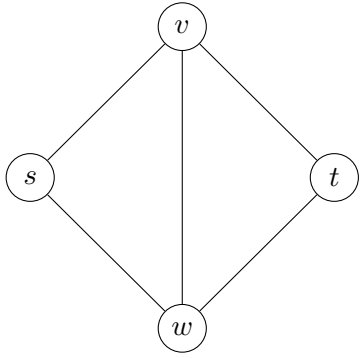

\begin{definition}
	Let $K_4$ be the complete undirected graph on 4 vertices. The graph that results from deleting one edge from $K_4$ is called \emph{diamond graph} (see \cref{fig:wheatstone-graph} and compare \cref{fig:DC-degenerate}).

	We call a graph $G$ a \emph{subdivision} of a graph $H$, if $G$ can be obtained by adding an arbitrary number of extra vertices along the edges of $H$. Furthermore, $H$ is a \emph{topological minor} of $G$ if $G$ contains a subdivision of $H$ as a subgraph (see, \eg, \cite{Diestel:2017}).

	An undirected graph $G$ is called \emph{cactus} if it does not contain the diamond graph as a topological minor \cite{El-Mallah:1988}. We say that an antisymmetric directed graph $(V,E)$ is a cactus if the graph that results from replacing all edges in $E$ by undirected edges is a cactus.
	\end{definition}

	The diamond graph is known in the context of electrical circuits as the \emph{Wheatstone bridge}, a device to measure the electrical parameters of a circuit (\cite{Wheatstone:1843}, originally introduced by \textcite{Christie:1833}). The structure is characterized by the fact that if a voltage is applied between the vertices $s$ and $t$, then the direction of the current between $v$ and $w$ changes in response to the resistance of the individual branches.

 The following characterization can be found in \cite{El-Mallah:1988}:
	\begin{proposition}
	\label{thm:cactus_alt_def}
	The following are equivalent:
	\begin{enumerate}
		\item The graph $(V,E)$ is a cactus.
		\item Every induced subgraph of $(V,E)$ which is $2$-connected\footnote{A graph is $2$-connected if it contains at least 3 vertices and remains connected whenever any single vertex (and all incident edges) is removed (see, \eg, \cite{Diestel:2017}).} is a simple cycle (i.e.~a connected graph where every vertex has degree 2).
		\item No edge $e \in E$ belongs to more than one simple cycle.
	\end{enumerate}
	\end{proposition}

	The following theorem proves that cacti are precisely those graphs on which every network is guaranteed to be non-degenerate:

	\begin{theorem}\label{thm:chord-free_nondeg}\label{thm:chord_possdeg}
		A network $G=(V,E,b,\pmax,\fmax)$ is non-degenerate for all choices of $b,\pmax,\fmax$ if and only if $(V,E)$ is a cactus.
	\end{theorem}

	\begin{proof}
		We first prove the \enquote{if}-part of the statement by induction over the number $|V|$ of vertices in $V$.
		If $|V| = 2$, then the statement is true by \cref{lem:K2-nondeg-rep}. Suppose that the statement is true for all graphs with less than $n$ vertices and let $G=(V,E,b,\pmax,\fmax)$ be a network such that the graph $(V,E)$ is a cactus graph with $|V| = n$.

		If on the one hand $(V,E)$ contains a cut-vertex $v^*$ such that the subgraph induced by $V \setminus \Set{v^*}$ consists of two sets of vertices $V_1$ and $V_2$ which are disconnected, then we know by definition that the subgraphs induced by $V_1 \cup \Set{v^*}$ and $V_2 \cup \Set{v^*}$ are cacti, as well (they cannot contain a topological minor that the original graph does not contain). Thus, the statement is true by induction using \cref{thm:1-con-nondeg}.

    If on the other hand $(V,E)$ does not contain such a cut-vertex, then $(V,E)$ is $2$-connected and, since $(V,E)$ is a cactus, this means by \cref{thm:cactus_alt_def} that it is a simple cycle and the statement follows immediately from \cref{lem:cycle-nondeg-rep}.

	\begin{figure}
		\centering
	\begin{tikzpicture}[decoration=snake, segment length=0.68cm, yscale=0.8]
		\node[gridnode, very thick] (v) at (0,0) {$v$};
		\node[gridnode, very thick] (w) at (5,0) {$w$};
		\node[gridnode] (w1) at (3.5,0) {};
		\node[gridnode] (t1) at (3,2.6) {};
		\node[gridnode] (s) at (-0.5,-2.5) {};
		\node[gridnode] (t) at (2.5,4) {};
		\node[gridnode] (t2) at (0.5,3) {};
		\node[gridnode] (vt) at (-1.5,2) {$v_t$};
		\node[gridnode] (v2t) at (1,2) {};
		\node[gridnode] (wt) at (6.5,2) {$w_t$};
		\node[gridnode] (vs) at (-1.5,-2) {$v_s$};
		\node[gridnode] (ws) at (6.5,-2) {$w_s$};
		\node[gridnode] (w2) at (7,0) {};
		\node[gridnode] (t3) at (-1.1,4) {};
		\node[gridnode] (s1) at (8.1,-1) {};

		\draw
			(v)	edge (vt)
				edge[accentuating green, decorate, very thick] (v2t)
				edge (vs)
				edge[accentuating green, decorate, very thick] node[midway, above] {$W$} (w1);

		\draw
			(t2)edge[accentuating orange, decorate, very thick] (vt)
				edge (v2t)
				edge[accentuating orange, decorate, very thick] (t3);

		\draw
			(w)	edge (wt)
				edge (ws)
				edge[accentuating green, decorate, very thick] (w1);

		\draw
			(w1)edge[accentuating green, decorate, very thick] (t1);

		\draw
			(t)	edge[accentuating orange, decorate, very thick] (t2)
				edge[accentuating orange, decorate, very thick, bend left] node[midway, above right] {$T$} (wt)
				edge[] (t1);
		\draw
			(vs)edge[accentuating light blue, decorate, very thick] (s);

		\draw
			(s)	edge[accentuating light blue, decorate, very thick, bend right] node[midway, above] {$S$} (ws)
				edge (v);

		\draw
			(w2)edge[accentuating light blue, decorate, very thick] (ws)
				edge[] (wt);

		\draw
			(ws)edge[accentuating light blue, decorate, very thick] (s1);

	\end{tikzpicture}
	\caption{Given a graph that contains the diamond graph as a topological minor, we select three sets of edges $S$, $T$ and $W$ that cover all vertices as shown in the figure. We can now choose capacities and edge weights in such a way that we obtain a network that is degenerate: Thus, a solution $f$ conforming with the $\alpha$-tree $(S \cup T \cup W, \Set{v,w})$ (shown in bold) need not be extremal.}
	\label{fig:wheatstone-minor}
	\end{figure}
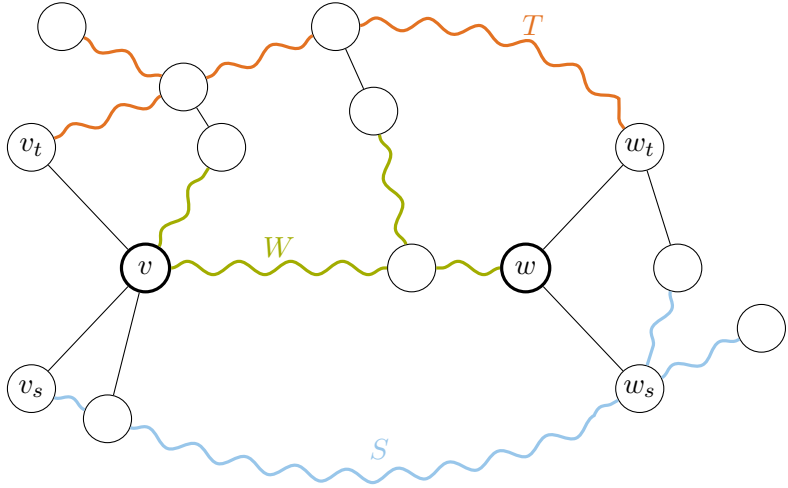

	For the \enquote{only-if}-part, assume that $(V,E)$ is not a cactus graph and thus contains the diamond graph as a topological minor.
		Let $H=(V,E')$ denote a minimal subgraph of $(V,E)$ that still contains the diamond graph as a topological minor. Then, $H$ is a subdivision of the diamond graph. Denoting the two vertices of degree 3 in $H$ by $v, w$,  this means $H$ contains three vertex-disjoint paths connecting $v$ and $w$, at least two of which have length at least 2 (\cf~\cref{fig:wheatstone-graph}). Choose two such paths $P_s$ and $P_t$ and on each of these, denote the intermediate vertex directly adjacent to $v$ by $v_s$ and $v_t$, respectively. Analogously, denote the intermediate vertex directly adjacent to $w$ by $w_s$ and $w_t$ (see \cref{fig:wheatstone-minor}). Note that it may be the case that $v_s = w_s$ and/or $v_t = w_t$.

		Denote by $P_{vw}$ the path between $v$ and $w$ in $H$ which does not contain any of $v_s,w_s,v_t,w_t$. As $H$ is a subgraph of $(V,E)$, the graphs $(V,P_{vw})$, $(V,P_s)$ and $(V,P_t)$ are equally subgraphs of $(V,E)$.

		We now build three sets $W \supset P_{vw}$, $S \supset P_s$ and $T \supset P_t$ such that every vertex in $V$ is covered by exactly one of the three sets as follows:
		Let $W \subset E$ be the set of edges of a maximal tree which contains $P_{vw}$ and does not cover any vertex that is already covered by $P_s$ or $P_t$. Analogously, let $S$ be the set of edges of a maximal tree which contains $P_s$ and does not cover any vertex already covered by $W$ or $P_t$. Finally, let $T$ be the set of edges of a maximal tree which contains $P_t$ and does not cover any vertex already covered by $W$ or $S$.

		We obtain three trees $S$, $T$, and $W$ with $v_s,w_s \in S$, $v_t,w_t \in T$, and $v,w \in W$, which cover all vertices in $G$. Now, let $E_{vs}$, $E_{ws}$, $E_{vt}$, $E_{wt}$ denote the sets of edges in $G$ that connect $v$ or $w$ with nodes covered by $S$ and $T$, respectively. We set $\fl_e=\fu_e=0$ for all edges $e$ in $W \cup S \cup T$ and $-\fl_e=\fu_e=\infty$ for all other edges. Similarly, we set $\pl_v=\pu_v=\pl_w=\pu_w = 0$ and $-\pl_{v'}=\pu_{v'}=\infty$ for all $v' \in V \setminus \Set{v,w}$. Now, let $f_e = 0$ for all $e \in E$. Then, $(W \cup S \cup T, \set{v,w})$ is an $\alpha$-tree in the resulting network that conforms with $f$.
		At the same time, we can choose $b$ such that
		\begin{equation*}
			\frac{\sum_{e \in E_{wt}} b_e}{\sum_{e \in E_{ws}} b_e}=\frac{\sum_{e \in E_{vt}} b_e}{\sum_{e \in E_{vs}} b_e} = 1.
		\end{equation*}
		Defining
		\begin{align*}
			\phi_{v'} = \begin{cases}
						0 & \text{if $v'$ is covered by $W$}\\
						1 & \text{if $v'$ is covered by $S$}\\
						-1 & \text{if $v'$ is covered by $T$}
						\end{cases}
		\end{align*}
		we have that $0 \neq \pm B^\top \phi \in \QDCf(G)$ since $(B^\top \phi)_e = 0$ for every edge $e \in W \cup S \cup T$. The only other bounds that are not $\pm \infty$ are the vertex constraints in $v$ and $w$.

	Let $N^{\inn}(v) := \Set{w \given (w,v) \in E}$ and $N^{\outn}(v) := \Set{w \given (v,w) \in E}$ for every $v \in V$.

		Then, for these constraints
		\begin{align*}
			(AB^\top \phi)_v &= \begin{multlined}[t]\sum_{\mathclap{v'\in N^{\inn}(v) \cap W}} b_{v'v}(\phi_v-\phi_{v'}) + \sum_{\mathclap{v'\in N^{\inn}(v) \cap T}} b_{v'v}(\phi_v-\phi_{v'}) + \sum_{\mathclap{v'\in N^{\inn}(v) \cap S}} b_{v'v}(\phi_v-\phi_{v'})\\
							- \sum_{\mathclap{v'\in N^{\outn}(v) \cap W}} b_{vv'}(\phi_{v'}-\phi_v) - \sum_{\mathclap{v'\in N^{\outn}(v) \cap T}} b_{vv'}(\phi_{v'}-\phi_v) - \sum_{\mathclap{v'\in N^{\outn}(v) \cap S}} b_{vv'}(\phi_{v'}-\phi_v)
							\end{multlined}\\
					&= \begin{multlined}[t] \sum_{v'\in N^{\inn}(v) \cap T} b_{v'v} \cdot 1 + \sum_{v'\in N^{\inn}(v) \cap S} b_{v'v}\cdot (-1)\\
							- \sum_{v'\in N^{\outn}(v) \cap T} b_{vv'}\cdot(-1) - \sum_{v'\in N^{\outn}(v) \cap S} b_{vv'}\cdot 1
							\end{multlined}\\
					&=\sum_{e \in E_{vt}} b_e - \sum_{e \in E_{vs}} b_e = 0
		\end{align*}
		and analogously $(AB^\top \phi)_w = 0$.
		Thus, $f = \nicefrac{1}{2} (B^\top\phi + (-B^\top\phi))$ is not extremal, which proves that the network $G$ is degenerate.
	\end{proof}

	The following fact is generally well-known. Since we could not find a suitable reference, however, we have added a short proof of our own, which provides us with a simple algorithm to determine wether non-degenerateness can be guaranteed for a given graph.

	\begin{proposition}
		It can be decided in linear time whether a graph $G=(V,E)$ is a cactus.
	\end{proposition}

	\begin{proof}
		Perform a depth-first search (DFS) in time $\mathcal{O}(|E|+|V|)$. Choose a back-edge $\Set{v,w}$, \ie an edge that connects $v$ to some previously-visited vertex $w$. Then, $w$ must be a predecessor of $v$, otherwise we would have visited $v$ from $w$ and the edge $\Set{v,w}$ would be in the DFS tree. Re-trace the path from $v$ back to $w$, marking edges as we go. If we reach a vertex $v'$ incident with another back-edge $\Set{v',w'}$, then recursively re-trace its cycle before continuing on the original path. If we reach the vertex $w$, then chose a new back-edge and repeat until all back-edges have been marked.

		If at some point we reach an edge that has already been marked, then $G$ was not a cactus graph (the edge belonged to two distinct simple cycles). Otherwise, $G$ is a cactus graph: Every simple cycle in $G$ produces a back-edge, we have marked every back-edge together with the cycle that it belongs to and no edge has been marked twice. Hence, no edge belongs to two simple cycles and thus by \cref{thm:cactus_alt_def}, the graph $G$ is a cactus.
	\end{proof}

	We have already encountered the diamond graph, the forbidden minor which characterizes cacti, in \cref{ex:DC_nondeg_fails}: It was our first example of a network that is degenerate. \Cref{thm:chord-free_nondeg} now shows that this is in fact the characterizing structure for non-degenerateness, at least from a topological point of view. Cactus graphs, for which we can guarantee non-degenerateness under all circumstances, are closely related to several results from the context of electrical circuits and related areas:

	They represent exactly those graphs that, if we add a single edge, remain \emph{of series parallel type} (or \emph{confluent}) in the sense of \cite{Duffin:1965}. A graph is of series parallel type if it can be constructed by iteratively replacing an edge by two sequential or two parallel edges, starting from an edge connecting a single vertex with itself. 
	Such networks are those, for which the equivalent resistance can be computed by iteratively applying Ohm's laws for resistors in series and parallel.
Note that the diamond graph itself is also of series parallel type, but by adding a single edge we can obtain $K_4$, which is no longer of series parallel type.

\section{Beyond Cactus Graphs} \label{sec:beyond}

	We can conclude that the well-known (but limited) family of cactus graphs completely characterizes those networks on which optimization problems over electrical power flows (as represented by the differential flow polyhedron) are \enquote{well-behaved} from a \emph{purely topological perspective}. However, network topology is only one aspect of non-degenerateness, the exact values of the vector $b$ of edge weights highly matter, as well: Indeed, all networks are non-degenerate unless the weight vector $b$ is chosen from the union of finitely many lower-dimensional subspaces of the parameter space (one subspace for every diamond minor in the network). The network from \cref{ex:DC_nondeg_fails}, for instance, is almost always non-degenerate unless $\frac{b_{wt}}{b_{ws}}=\frac{b_{vt}}{b_{vs}}$. In the theory of electrical circuits, this case is known as the one where the Wheatstone bridge is \emph{balanced} \cite{Duffin:1965}.
Furthermore, note that by \cref{thm:alpha-tree-ex}, even in this case, every extremal point corresponds to an $\alpha$-tree (only the reverse might not always hold).

	This implies that we can in practice expect non-degenerateness to hold not only for the (quite limited) family of networks  with cactus topology and every network can be made non-degenerate by slightly disturbing the edge weights.

  For a specific network, however, it is generally \NP-complete to determine whether or not it is non-degenerate, as we will see in Section \ref{sec:hardness}.
  Still, we can identify \emph{sufficient conditions} to identify $\alpha$-trees for which every conforming differential flow is always extremal (regardless of whether or not the underlying network is non-degenerate).

\subsection{Sufficient Conditions} \label{sec:sufficient}
In this subsection, we present some conditions for a given $\alpha$-tree under which a solution $f \in \QDCf(G)$ that conforms to it is guaranteed to be extremal. We start by observing the following:

\begin{remark}\label{rem:f_extremal_g_zero}
	Let $G=(V,E,b,\pmax,\fmax)$ be a network and $f \in \QDCf(G)$. Then $f$ is \emph{not} extremal if and only if there exists a differential flow $g \neq 0$ (which itself does not need to be feasible) such that $f \pm g \in \QDCf(G)$. In particular, for every vertex $v$ with $-(Af)_v \in \Set{\pl,\pu}$, we must have $A_v g = 0$ and for every edge $(v,w) \in E$ with $f_{vw} \in \Set{\fl_{vw},\fu_{vw}}$, we must have $g_{vw} = 0$.
\end{remark}

In light of the remark above, it makes sense to investigate the restrictions which an $\alpha$-forest $F=(E_F,V_F)$ that conforms with $f$ imposes on a differential flow $g$ with $f \pm g \in \QDCf(G)$. For a potential $\phi$ that induces $g$, the above observation implies for every connected component $S$ of the graph $(V,E_F)$ and $v,w \in S$ that $\phi_v = \phi_w$ (since $f_{vw} \in \Set{\fl_{vw},\fu_{vw}}$ for every $(v,w) \in E_F$), \ie~$\phi$ is constant within any connected component of $(V,E_F)$.

In order to represent restrictions imposed by $F$ on $\phi$ between different connected components of $(V,E_F)$, we need the following definition of a \emph{generalized differential flow} which is induced by a potential $\phi$ via a \emph{generalized \tension vector} $b$.

\begin{definition}[Generalized Differential Flow]
	Let $V$ be a finite set and $b \in \posreals^{V \times V}$. A potential $\phi \in \reals^V$ induces a \emph{generalized differential flow} $f \in \reals^{V \times V}$ on $V$ with respect to $b$ by $f_{vw} = b_{vw} (\phi_w-\phi_v)$. We say that $f$ is feasible if
	\begin{equation}
		\sum_{w \in V} f_{vw} = 0\label{eq:gen_diff_flow}
	\end{equation}
	for all vertices $v \in V$.
\end{definition}

Note that in contrast to the \tension vector in a network, a generalized \tension vector is defined for all pairs $(v,w) \in V \times V$ (even if $(v,w) \notin E$). In particular, there may be pairs $(v,w)$ with $b_{vw} \neq b_{wv}$ and, as a consequence, $f_{vw} \neq -f_{wv}$ (one could even be 0).

Furthermore, equation \cref{eq:gen_diff_flow} is similar to a common flow conservation constraint, but not identical: For every $v \in V$, we only sum over the \enquote{outgoing} entries $f_{vw}$ for $w \in V$ (which may be positive or negative). In particular, the \enquote{incoming} flow value $f_{wv}$ does not appear in the constraint  \cref{eq:gen_diff_flow} for vertex $v$ but only in the corresponding constraint for vertex $w$.

Note further that for any $\phi_0 \in \reals$, the trivial potential $\phi \equiv \phi_0$ induces the generalized differential flow $f \equiv 0$ (independently of $b$), which is always feasible. Analogously to \cref{def:dc_ext_forest}, the following tree structure can be used to characterize the cases where $f \equiv 0$ is the unique feasible generalized flow (as we will see in \cref{lem:gendiffflow_trivsol}):

\begin{definition}
	Let $T=(V,E)$ be a weakly connected directed graph that does not contain an undirected cycle. We then call $T$ a \emph{directed tree}. If furthermore $|\delta^{\outn}(v)| \leq 1$ for all $v \in V$, then $T$ is an \emph{anti-arborescence}.
\end{definition}

In particular, if $T$ is an anti-arborescence, then there is a unique vertex $v_0 \in T$ such that $|\delta^{\outn}(v)| = 0$ (the \emph{sink} of the anti-arborescence).

\begin{lemma}
	\label{lem:gendiffflow_trivsol}
	Let $V$ be a finite set, $b \in \posreals^{V \times V}$, and $E := \Set{(v,w) \in V \times V \given v \neq w \text{ and } b_{vw} > 0}$ such that $(V,E)$ contains a spanning anti-arborescence as a subgraph. If $\phi$ induces a feasible generalized differential flow $f$ on $V$ with respect to $b$, then $\phi \equiv \phi_0$ for some $\phi_0 \in \reals$ and as a consequence $f \equiv 0$.
\end{lemma}

\begin{proof}
	Let $\phi \in \reals^V$ be a potential that induces a feasible generalized differential flow. Let $v_0 \in V$ be the sink of the anti-arborescence $F$. We prove that $\max_{v \in V} \phi_v = \phi_{v_0}$, the argument to see that $\min_{v \in V} \phi_v = \phi_{v_0}$ is identical. Let $v_1 \in \argmax_{v \in V} \phi_v$, \ie $\phi_{v_1} \geq \phi_w$ for all $w \in V$. If $\phi_{v_1} = \phi_{v_0}$, then we are done. Otherwise, in particular $v_1 \neq v_0$ and, since $v_0$ is the unique sink, there exists $v_2 \in V$ with an edge $(v_1,v_2) \in T$. By the definition of a feasible generalized differential flow and by maximality of $\phi_{v_1}$, it follows
	\begin{equation*}
		0 = \sum_{w \in V\setminus\Set{ v_1 }} \underbrace{\underbrace{b_{v_1w}}_{\geq 0} \underbrace{(\phi_w-\phi_{v_1})}_{\leq 0}}_{\leq 0}.
	\end{equation*}
	Hence, in particular $b_{v_1v_2} (\phi_{v_2}-\phi_{v_1}) = 0$. Since furthermore $(v_1,v_2) \in E$, we have that $b_{v_1v_2} > 0$ which implies that $\phi_{v_2}=\phi_{v_1} > \phi_{v_0}$ and thus in particular $v_2 \neq v_0$. Hence, $v_2 \in \argmax_{v \in V} \phi_v$ and we can apply the same argument to find a vertex $v_3$ with $\phi_{v_2} = \phi_{v_3}$. Since $T$ is an anti-arborescence and we only traverse edges in $T$ in the direction of the sink, we never reach the same vertex twice. By iterating this argument, we therefore finally obtain that $\phi_{v_1} = \phi_{v_2} =  \phi_{v_3} = \dotsb = \phi_{v_0}$, a contradiction.
\end{proof}

The above lemma can be used to derive a sufficient condition for extremality of a differential flow $f \in \QDCf(G)$ based on the existence of an $\alpha$-tree $F$ with one additional property (which in particular excludes the situation encountered in \cref{ex:DC_nondeg_fails}).

\begin{theorem}\label{thm:suff_cond}
	Let $G=(V,E,b,\pmax,\fmax)$ be a network, $f \in \QDCf(G)$ and $F=(E_F,V_F)$ an $\alpha$-tree in $G$ that conforms with $f$. Furthermore, let $F$ be such that every connected component of the graph $(V,E_F)$ contains at most one vertex which is active in $F$. Then, $f$ is extremal.
\end{theorem}

\begin{proof}
	In order to show that $f$ is extremal, we need to prove that no non-zero differential flow $g$ exists such that $f\pm g \in \QDCf(G)$ (see \cref{rem:f_extremal_g_zero}). Suppose otherwise and let $\phi \in \reals^V$ be a potential that induces $g$. Denote by $\mathcal{S}$ the set of connected components of the graph $(V,E_F)$.

	By our assumption, every connected component $S$ contains at most one vertex which is active in $F$. Denote this vertex (if it exists) by $v_S$. For two connected components $S,T \in \mathcal{S}$, we define the value $b'_{ST}$ as follows (we write $\delta(S) := \bigcup_{v \in S} \delta(v)$ for any $S \subseteq V$):
	\begin{align*}
		b'_{ST} &:= \begin{cases}
					\sum_{e \in \delta(v_S) \cap \delta(T)} b_e & \text{if $S$ contains a vertex $v_S$ active in $F$}\\
					0 & \text{else.}
					\end{cases}
	\end{align*}
	Note that it may be the case that $b'_{ST} \neq b'_{TS}$, in particular one of the two may be $0$ (or both), but in any case $b' \geq 0$ since $b > 0$.

	Let $\alpha_F$ be a vertex-orientation map for $F$ and denote by $F'$ the set of pairs of connected components for which there exists an active vertex in the first component which is mapped by $\alpha_F$ to an edge connecting it with the second component:
	\begin{align*}
		F' &:= \Set*{(S,T) \in \mathcal{S}\times\mathcal{S} \given \exists (v,w) \in \alpha_F(V_F \cap S): (v,w) \in (S \times T) \cup (T\times S)}
	\end{align*}

	By the definition of $\alpha$-trees, the set $E_F \cup \alpha_F(V_F)$ does not contain an undirected cycle and since $F$ is maximal, it spans $V$. This implies that the graph $(\mathcal{S},F')$ is weakly connected and does not contain a cycle, either. Furthermore, since every connected component $S \in \mathcal{S}$ contains at most one active vertex, there is at most one $T \in \mathcal{S}$ such that $(S,T) \in F'$ for every $S \in \mathcal{S}$. The graph $(\mathcal{S}, F')$ is thus an anti-arborescence.

	Returning to the differential flow $g$, \cref{rem:f_extremal_g_zero} implies for any $v,w \in V$ with $(v,w) \in E_F$ that $g_{vw} = 0$ and hence $\phi_v=\phi_w$.
	Let $\phi' \in \reals^{\mathcal{S}}$ be defined by $\phi'_S := \phi_v$ for any $S \in \mathcal{S}$ and $v \in S$ (note that this is well-defined, since $\phi_v=\phi_w$ for any $S \in \mathcal{S}$ and $v,w \in S$).
	Now, observe that $\phi'$ induces a generalized differential flow on the set $\mathcal{S}$ with respect to $b'$. Thus, for all $S \in \mathcal{S}$, one of the following two statements holds:
	\begin{itemize}
		\item Either $S$ contains no active vertex and then $b'_{ST} = 0$ for all $T \in \mathcal S$ and hence $\sum_{T \in \mathcal{S}\setminus\Set{S}} b'_{ST} (\phi_T - \phi_S) = 0$, or
		\item $S$ contains exactly one active vertex $v_S$ and
		\begin{align*}
			\sum_{T \in \mathcal{S}\setminus\Set{S}} b'_{ST} (\phi_T - \phi_S) &= \sum_{T \in \mathcal{S}\setminus\Set{S}} \sum_{e \in \delta(v_S) \cap \delta(T)} b_e (\phi_T - \phi_S)\\
			&= \sum_{T \in \mathcal{S}\setminus\Set{S}} \left(\ \quad\sum_{\mathclap{\substack{w \in\\N^{\outn}(v_S) \cap T}}}\ b_{v_Sw} (\phi_w - \phi_{v_S}) - \ \sum_{\mathclap{\substack{w \in\\N^{\inn}(v_S) \cap T}}} \ b_{wv_S} (\phi_{v_S} - \phi_w)\right)\\
			&= \sum_{w \in N^{\outn}(v_S)} b_{v_Sw} (\phi_w - \phi_{v_S}) - \sum_{w \in N^{\inn}(v_S)} b_{wv_S} (\phi_{v_S} - \phi_w)\\
			&= \sum_{w \in N^{\outn}(v_S)} g_{v_Sw} - \sum_{w \in N^{\inn}(v_S)} g_{wv_S} = 0
		\end{align*}
		 where the last equality follows from \cref{rem:f_extremal_g_zero}.
	\end{itemize}

	Hence, the generalized differential flow induced by $\phi'$ is feasible.

	Let $E' := \Set{(S,T) \in \mathcal{S} \times \mathcal{S} \given S \neq T \text{ and } b'_{ST} > 0}$ and observe that $F' \subseteq E'$, since the definition of $F'$ captures exactly those pairs $(S,T)$ for which the first case in the definition of $b'$ applies.
	Hence, the anti-arborescence $(\mathcal{S}, F')$ is a subgraph of $(\mathcal{S}, E')$ and we now obtain from \cref{lem:gendiffflow_trivsol} that there exists $\phi_0 \in \reals$ such that $\phi' \equiv \phi_0$ and thus $\phi \equiv \phi_0$ which implies that $g\equiv 0$.
\end{proof}

Using \cref{thm:suff_cond}, we can prove another sufficient condition, which excludes a different aspect of \cref{ex:DC_nondeg_fails}. In particular, this also implies that every network with maximal degree $\leq 2$ is non-degenerate, providing an alternative proof for \cref{lem:K2-nondeg-rep}, \cref{cor:tree-nondeg-rep}, and \cref{lem:cycle-nondeg-rep}.

\begin{theorem}\label{thm:suff_cond_small_deg}
	Let $G=(V,E,b,\pmax,\fmax)$ be a network, $f \in \QDCf(G)$ and $F=(E_F,V_F)$ an $\alpha$-tree in $G$ that conforms with $f$. Furthermore, let $F$ be such that there is at most one $v \in V_F$ with $\deg{v} \geq 3$. Then, $f$ is extremal.
\end{theorem}

\begin{proof}
	We prove the statement by induction over the number $|V_F|$ of active vertices in $F$. If $|V_F|=0$, \ie there is no vertex active in $F$, then $(V,E_F)$ is in fact a spanning tree and, since any potential $\phi$ that induces a differential flow $g$ with $f \pm g \in \QDCf(G)$ has to be constant within connected components of $(V,E_F)$, $f$ is extremal.

	Now, let $|V_F|=k$. We distinguish two cases:
	First, assume that $V_F$ contains no vertex of degree at most $2$ that is incident with an active edge. Then, for every active vertex $v$, one of the following holds with respect to the connected component $S_v$ of the graph $(V,E_{F})$ that contains $v$:
	\begin{enumerate}
		\item $v$ is not incident with an active edge and hence $S_v$ is a singleton, or
		\item $v$ is incident with an active edge, which implies that $\deg{v} \geq 3$ and, as there can be only one such vertex, $v$ is the only active vertex in $S_v$.
	\end{enumerate}
	In both cases, $v$ is the only active vertex in $S_v$ and hence every connected component of $(V,E_F)$ contains at most one active vertex. Thus, $F$ satisfies the conditions of \cref{thm:suff_cond} and $f$ is extremal.

	Now, assume that $V_F$ does contain a vertex $v$ of degree at most $2$ that is incident with an active edge $e^* \in E_F$ and let $\alpha_F: V_F \rightarrow E \setminus E_F$ be a vertex-orientation map for $F$.

	Then $e^* \in E_F$ and $\alpha_F(v) \in E \setminus E_F$ are the only edges incident with $v$.
	Assume \Wlog~that both are oriented away from $v$ and let $s,t \in V$ be such that $\alpha_F(v)=(v,t)$ and $e^* = (v,s)$.
	Let $p^*_v := \sum_{e \in \delta^\outn(v)} f_e - \sum_{e \in \delta^\inn(v)} f_e = f_{vt}+f_{vs}$.

	Since $(v,s)$ and  $v$ are active, it follows $f_{vs} \in \Set{\fu_{vs},\fl_{vs}}$ and $p^*_v \in \Set{\pu_v,\pl_v}$.
	 Thus, for every $f',f'' \in \QDCf(G)$ such that $f$ is a convex combination of $f'$ and $f''$, it must hold that $f'_{vs} = f''_{vs} =f_{vs}$ and $f'_{vs} + f'_{vt} = f''_{vs} + f''_{vt} =f_{vs} + f_{vt}$. Together, this also implies $f'_{vt} = f''_{vt} = p^*_v-f_{vs}$.

	 Let $\fmax'$ be given by
	\begin{equation*}
		\fmax'_e = \begin{cases}
						(f_{vt},f_{vt}) & \text{if } e=(v,t)\\
						\fmax_e & \text{else}
					\end{cases}
	\end{equation*}
	and define $H:=(V,E,b,\pmax,\fmax')$. Then, we can conclude that $f$ is extremal in $\QDCf(G)$ if and only if $f$ is extremal in $\QDCf(H)$.

	But in $H$, $F' := (E_F \cup \Set{(v,t)}, V_F \setminus \Set{v})$ is an $\alpha$-tree that conforms to $f$. Since $|V_{F'}| < |V_F|$, we can conclude by induction that $f$ is extremal in $\QDCf(H)$ and therefore also in $\QDCf(G)$.
\end{proof}

\subsection{Non-degenerateness is \NP-hard} \label{sec:hardness}

In Section \ref{sec:classes_of_nondeg_graphs}, we have identified all graphs that are non-degenerate, regardless of the values $b$, $\pmax$, and $\fmax$ in a network derived from those graphs. Section \ref{sec:sufficient} collects some sufficent conditions for when a given $\alpha$-tree definitely identifies an extremal point.

We conclude the paper by noting that in general it is \NP-hard (see \cite{Garey:2002}) to determine whether in a given network every $\alpha$-tree identifies an extremal point.

	\begin{theorem}\label{thm:rep_prop_hard}
		Given a network $G=(V,E,b,\pmax,\fmax)$, it is \NP-complete to decide whether there exists and $\alpha$-tree $F$ and a non-extremal point $f \in \QDCf(G)$ that conforms with $F$.
	\end{theorem}

	\begin{proof}
		We first settle membership in \NP: Suppose that there exists a non-extremal solution $f \in \QDCf(G)$ that conforms to the $\alpha$-tree $F=(E_F,V_F)$. Let $P$ denote the set of vertices of $\QDCf(G)$ that are contained in the minimal face of $\QDCf(G)$ which contains $f$ and let $f' := \nicefrac{1}{|P|} \sum_{p \in P} p$. Then, $f'$ conforms to $F$ as well and is not extremal (since $|P| \geq 2$). Given a certificate in the form of $f'$, $F$ and a corresponding vertex-orientation map $\alpha_F: V_F \rightarrow E \setminus E_F$ (which are all polynomial in size), we can easily check that $F$ is indeed an $\alpha$-tree and that $f'$ conforms with $F$. Finally, non-extremality of $f'$ can be verified in polynomial time via the rank of the matrix of active constraints.

		To prove the hardness, we provide a reduction from the \NP-complete problem \textsc{SubsetSum} \cite{Garey:2002}. Given a number $n \in \naturals$ and sizes $\alpha_i \in \naturals$ for each $i \in [n]$, as well as $\beta \in \naturals$, decide whether there is a subset $I \subseteq [n]$ with $\sum_{i \in I} \alpha_i = \beta$.

		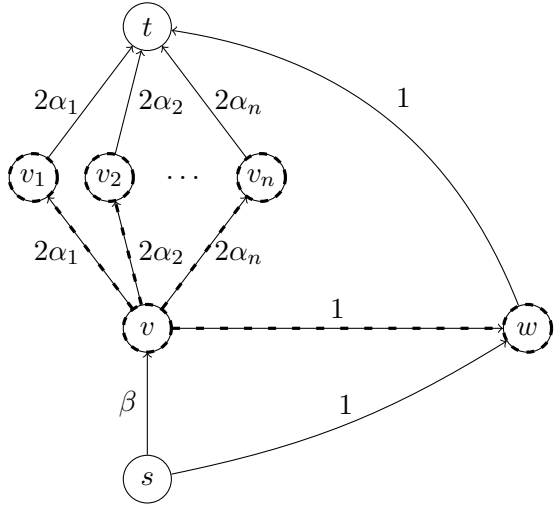
\begin{figure}
			\centering
		\begin{tikzpicture}
			\node[gridnode] (v1) at (0,0) {$v$};
			\node[gridnode] (v2) at (5,0) {$w$};
			\node[gridnode] (s) at (0,-2) {$s$};
			\node[gridnode] (t) at (0,4) {$t$};
			\node[gridnode] (a1) at (-1.5,2) {$v_1$};
			\node[gridnode] (a2) at (-0.5,2) {$v_2$};
			\node (a3) at (0.5,2) {\dots};
			\node[gridnode] (a4) at (1.5,2) {$v_n$};

			\node[gridnode, very thick, dashed] at (v1) {};
			\node[gridnode, very thick, dashed] at (v2) {};
			\node[gridnode, very thick, dashed] at (a1) {};
			\node[gridnode, very thick, dashed] at (a2) {};
			\node[gridnode, very thick, dashed] at (a4) {};

			\draw
				(v1)	edge[->] node[left] {$2\alpha_1$} (a1)
						edge[->] node[right] {$2\alpha_2$} (a2)
						edge[->] node[right] {$2\alpha_n$} (a4)
						edge[->] node[above] {$1$} (v2);

			\draw[loosely dashed, very thick]
				(v1)	edge (a1)
						edge (a2)
						edge (a4)
						edge (v2);

			\draw
				(t)	edge[<-] node[left] {$2\alpha_1$} (a1)
					edge[<-] node[right] {$2\alpha_2$} (a2)
					edge[<-] node[right] {$2\alpha_n$} (a4)
					edge[<-, bend left] node[above right] {$1$} (v2);
			\draw (s) edge[->] node[left] {$\beta$} (v1)
					edge[->, bend right=10] node[above] {$1$} (v2);
		\end{tikzpicture}
		\caption{Solving the well-known \NP-complete problem \textsc{SubsetSum} is equivalent to determining whether the shown network (the edge weights represent the edge \tension $b$) is non-degenerate.}
		\label{fig:rep_prop_hard}
		\end{figure}

		Consider the network depicted in \cref{fig:rep_prop_hard} where the shown edge weights denote the edge \tension $b$, thickly dashed edges have an upper capacity bound of $1$, the thickly dashed vertices have the following upper and lower bounds: $\pmax_v = (-\beta,\beta)$ and $\pmax_w = \pmax_{v_1} = \dotsb = \pmax_{v_n} = (0,\infty)$. All other edges and vertices have infinite upper and lower bounds. We prove that the \textsc{SubsetSum} instance is a \emph{yes}-instance if and only if the network is degenerate.

		Let $f$ be a flow and $F=(E_F,V_F)$ an $\alpha$-tree which conforms to $f$. Let $\alpha_F: V_F \rightarrow E \setminus E_F$ be a vertex-orientation map for $F$. We start by observing that by \cref{thm:suff_cond_small_deg}, $f$ is extremal unless both vertices $v$, $w$ are active in $F$.
		Now, let $\phi$ be a potential such that $f \pm B^\top \phi$ is feasible. This means in particular that $B^\top \phi$ disappears on all inequalities active in $F$. W.l.o.g., let $\phi_{v} := 0$. Since $f = \nicefrac{1}{2} (f + B^\top \phi + f - B^\top \phi)$, it holds that $f$ is extremal if and only if $\phi \equiv 0$ is the only possible solution.

		We now show that if for one of the vertices $v_i$, both the vertex and the edge connecting it to $v$ are active, then $f$ is extremal. As the edge $(v,v_i)$ is active, we have that $\phi_{v_i}=0$. Similarly, as $v_i$ is active, $\phi_t = 0$. We distinguish two cases:
		\begin{enumerate}
			\item The edge $(v,w)$ is active. In this case, it follows that $\phi_{w}=\phi_{v}=0$ and hence $\phi_{s}=0$ by activity of $w$. As $F$ is an $\alpha$-tree, for every vertex $v_j$ at least one of the two edges incident with $v_j$ must be contained in $E_F \cup \alpha_F(V_F)$. Since $F$ conforms with $f$, we have that for \emph{all but one} of these, either $v_j$ is active or the edge $(v,v_j)$ (the remaining vertex can be connected by $v$). In both cases, $\phi_{v}=\phi_t=0$ implies that $\phi_{v_j}=0$. For the final vertex $v_{j^*}$, the same now follows by activity of $v$.
			\item The edge $(v,w)$ is not active. Then either $\alpha_F(v) = (s,v)$ or $\alpha_F(w) = (s,w)$, since otherwise $s$ would be disconnected. Furthermore, if $\alpha_F(w) = (s,w)$ then $\alpha_F(v) = (v,w)$, since otherwise the pair $\Set{s,w}$ would be disconnected (as neither $t$ nor $(w,t)$ can be active). Together, we have $\alpha_F(v) \in \Set{(s,v),(v,w)}$. As above, for every vertex $v_j$ at least one of the two edges incident with $v_j$ must be contained in $E_F \cup \alpha_F(V_F)$. Since $F$ conforms with $f$ and $\alpha_F(v) \in \Set{(s,v),(v,w)}$, we have that for \emph{all} of these, either $v_j$ is active or the edge $(v,v_j)$. In both cases, $\phi_{v}=\phi_t=0$ implies that $\phi_{v_j}=0$. Now, suppose that $\phi_{w}>\phi_{v}$. Then, activity of $v$ implies that $\phi_{s} < \phi_{v}$. At the same time, activity of $w$ implies that $\phi_s > \phi_{w}$, a contradiction. Therefore, $\phi \equiv 0$.
		\end{enumerate}
		We have concluded that $\phi \equiv 0$ if for one of the vertices $v_i$, both the vertex and the edge connecting it to $v$ are active.

		If $f$ is \emph{not} extremal, it must therefore hold that for all vertices $v_i$, only the vertex itself or the edge connecting it to $v$ can be active. At the same time, as $F$ is an $\alpha$-tree it has to hold that $|F|=n+3$. This implies that the edge $(v,w)$ is active and for all vertices $v_i$, \emph{exactly} one of the two (the vertex itself or the edge connecting it to $v$) needs to be active.
		In this case, $\phi_{v}= 0$ implies $\phi_{w} = 0$ (by activity of $(v,w)$) and furthermore $\phi_{v_j}=0$ for all $v_j$ for which the edge $(v,v_j)$ is active, while for all $v_j$ that are active, it follows that $\phi_{v_j}=\nicefrac{1}{2} \cdot \phi_t$.

	 By activity of $w$, it follows that $\phi_s = -\phi_t$. 	Thus, $\phi_t =0$ would imply $\phi \equiv 0$, in contradiction to $f$ not extremal. Hence, $\phi_t \neq 0$ and \Wlog~we may assume $\phi_t > 0$.  The vector $\phi$ now satisfies all constraints that are active in $F$ if and only if the flow balance at $v$ is 0, \ie
		\begin{equation*}
			\phi_s \cdot \beta = \sum_{\text{$v_a$ active in $F$}} \frac{\phi_t}{2} \cdot 2\alpha_a= \phi_s \cdot \sum_{\text{$v_a$ active in $F$}} \alpha_a.
		\end{equation*}
		Such a selection of active vertices $v_j$ hence exists if and only if $(n,\alpha_1,\alpha_2,\dotsc,\alpha_n,\beta)$ is a \emph{yes}-instance of \textsc{SubsetSum}.
	\end{proof}

\section{Conclusion and Outlook}

We have investigated the polytope of \emph{differential flows} in a network, which captures a common way of representing electrical power flows in \emph{Energy System Optimization}. A characterization of extremal points in terms of so-called $\alpha$-trees has been derived, similar to the characterization of extremal points in classical network flow problems by spanning trees of edges with non-zero flow. It is proven that every extremal solution can be represented by a corresponding $\alpha$-tree (\cref{thm:alpha-tree-ex}) and that the reverse is also true in almost all cases. In particular, it is true regardless of the choice of network parameters for networks on cactus graphs (\cref{thm:chord-free_nondeg}). Finally, for a specific given network, we showed that it is \NP-hard to decide whether every feasible $\alpha$-tree indeed corresponds to an extremal solution (\cref{thm:rep_prop_hard}).

Our work may raise new theoretical questions, most notably in the area between networks on cactus graphs, which are always non-degenerate, and our hardness result: Can other classes of networks be identified (e.g., based on the elasticity vector $b$) for which non-degenerateness can always be guaranteed?

From a practical perspective in the context of power networks, our work might open two interesting directions of future work:
\begin{enumerate}
\item Algorithmically, one famous application of the characterization of extremal points in network flow problems is the \emph{Network Simplex Algorithm}, which allows for an extremely efficient solution of associated optimization problems. Since \cref{thm:alpha-tree-ex} guarantees that every extremal differential flow can be represented by an $\alpha$-tree, a similar approach might be possible for optimization problems over differential flows, as they widely appear in Energy System Optimization.
\item Structurally, differential flows which might not be feasible (see \cref{def:feas_diff_flow}) can be interpreted as power flows which are electrically feasible, but violate edge and/or vertex capacity constraints. This might be a useful concept in the study of \emph{Capacity Expansion} problems, which investigate the optimal increase in generation and network capacities to make given energy systems feasible.
\end{enumerate}


\printbibliography[heading=bibintoc]

\end{document}